\documentclass[reqno,12pt]{amsart}
\usepackage{graphicx,epsfig}
\usepackage{epstopdf}
\usepackage{pdfpages}
\usepackage[labelfont=up]{subcaption}
\usepackage{amssymb}
\usepackage{empheq}
\usepackage{cases}
\usepackage{amsthm,amsmath}
\usepackage{caption,lipsum}
\usepackage{stmaryrd}
\usepackage{tabularx}
\usepackage{color}
\usepackage{empheq,float}
\DeclareGraphicsExtensions{.eps}
\usepackage{amssymb,amsmath,graphicx,amsfonts,euscript,mathrsfs}
\usepackage{bm}
\usepackage{hyperref}
\usepackage{enumerate}

\newtheorem{theorem}{Theorem}[section]
\newtheorem{lemma}[theorem]{Lemma}
\newtheorem{proposition}[theorem]{Proposition}

\theoremstyle{definition}

\newtheorem{example}[theorem]{Example}

\newtheorem{assumption}[theorem]{Assumption}

\theoremstyle{remark}
\newtheorem{remark}{Remark}

\newcommand{\fe}{\mathrm{e}}

\newcommand{\bR}{{\mathbb R}}

\newcommand{\bT}{{\mathbb T}}
\newcommand{\bN}{{\mathbb N}}
\newcommand{\bZ}{{\mathbb Z}}

\numberwithin{equation}{section}

\usepackage[show]{ed}

\newcommand{\bxi}{\boldsymbol{\xi}}

\begin{document}
\title[Quasi-Monte Carlo for NLS {with random potentials}]{Error estimate of a quasi-Monte Carlo time-splitting pseudospectral method for nonlinear Schr\"odinger equation with {random potentials}}

\author[Z. Wu]{Zhizhang Wu}
\address{\hspace*{-12pt}Z.~Wu: Department of Mathematics, The University of Hong Kong, Pokfulam Road, Hong Kong, China.}
\email{wuzz@hku.hk}

\author[Z. Zhang]{Zhiwen Zhang}
\address{\hspace*{-12pt}Z.~Zhang: Department of Mathematics, The University of Hong Kong, Pokfulam Road, Hong Kong, China.}
\email{zhangzw@hku.hk}

\author[X. Zhao]{Xiaofei Zhao}
\address{\hspace*{-12pt}X.~Zhao: School of Mathematics and Statistics \& Computational Sciences Hubei Key Laboratory, Wuhan University, Wuhan, 430072, China}
\email{matzhxf@whu.edu.cn}
\urladdr{http://jszy.whu.edu.cn/zhaoxiaofei/en/index.htm}




\begin{abstract}\noindent
In this paper, we consider the numerical solution of a nonlinear Schr\"odinger equation with spatial random potential. The randomly shifted quasi-Monte Carlo (QMC) lattice rule combined with the time-splitting pseudospectral discretization is applied and analyzed. The nonlinearity in the equation induces difficulties in estimating the regularity of the solution in random space. By the technique of weighted Sobolev space, we identify the possible weights and show the existence of QMC that converges optimally at the almost-linear rate without dependence on dimensions. The full error estimate of the scheme is established. We present numerical results to verify the accuracy and investigate the {wave propagation}.
\end{abstract}

\maketitle

\section{Introduction}
In this paper, we consider the following one-space-dimensional nonlinear Schr\"{o}dinger (NLS) equation with a spatial random potential on the torus:
\begin{equation}\label{model origin}
\left\{
\begin{aligned}
&i\partial_t\psi=-\frac{1}{2}\partial_x^2\psi+V(\omega,x)
\psi+\alpha|\psi|^2\psi,\quad x\in \bT, \ \omega \in \Omega,\ t>0,\\
&\psi(t=0,\omega,x)=\psi_{in}(x),\quad x\in \bT,\ \omega \in \Omega.
\end{aligned}
\right.
\end{equation}
Here $t\geq0$ is the time variable, $x\in\bT$ is the space variable with $\bT$ the one-dimensional torus (periodic boundary condition in $x$), $\omega\in\Omega$ is the random sample with $\Omega$ the sample space, $\psi=\psi(t,\omega,x)$ is the complex-valued unknown with $\psi_{in}$ the given initial data, $V(\omega, x)$ is a given real-valued random potential, and $\alpha\in\bR$ is a given parameter describing the strength of the nonlinearity. {We refer to the NLS equation \eqref{model origin} as the random NLS.}

It is well known that in the one-space-dimensional case in the presence of a random potential and in the absence of nonlinearity {(i.e., $\alpha=0$)} {with probability one all the states are exponentially localized \cite{altmann2020quantitative,altmann2022localization,Anderson,filoche2012universal,lee1985disordered}}. {With $\alpha = 0$ the equation in \eqref{model origin} is a linear Schr\"{o}dinger equation with random potentials}, which belongs to the family of the Anderson models named after P. Anderson. In 1958, P. Anderson considered the discrete linear Schr\"odinger equations and discovered the localization of waves due to the disorder induced by the random potential \cite{Anderson}. Ever since then, the Anderson localization effect has been widely considered in many applications such as semiconductors and acoustic waves.

When $\alpha\neq0$, the interplay between disorder and nonlinear effects leads to new interesting physics.
In spite of the extensive research, many fundamental problems still remains open. A long-lasting question in mathematics and physics is whether a nonlinearity especially the defocusing nonlinearity, i.e., $\alpha>0$ in (\ref{model origin}), could break the localization \cite{soffer}. Many analytical and numerical investigations have been done to address this question. {The study of the long-time behavior of wave propagation so far still relies on numerical simulations, which are performed on the discrete NLS models \cite{prl2,Shepelyansky}.} Our model problem (\ref{model origin}) is the continuous version of the {random} NLS. It has been considered physically for modeling the Anderson localization of Bose-Einstein Condensates \cite{pra,PRLnew} and for the study of nonlinear dispersive wave dynamics in a disordered medium \cite{Conti,nonlinear wave,CAL2}. {For theoretical results, we refer readers to \cite{debussche3,debussche1} for the well-posedness results of (\ref{model origin}) and to \cite{Labbe,CAL1,Gu} for other related mathematical properties.} Here we consider the {one-space-dimensional} case and the torus domain in (\ref{model origin}) for simplicity.

In this work, we consider the parametrization of the random potential in the manner of the Karhunen-Lo\`{e}ve expansion \cite{Sloan NM,Sloan JCP,kuo2012quasi,Xiu}:
\begin{align}
V(\omega,x)=v_0(x)+\sum_{j=1}^{\infty}\sqrt{\lambda_j}\xi_j(\omega)v_j(x),\quad x\in \bT,\ \omega\in\Omega,
\label{random-potential origin}
\end{align}
where $v_ 0(x)$ is a deterministic function, $\{ v_j(x) \}_{j \geq1}$ are the physical components, $\lambda_1 > \lambda_2>\cdots > 0$ are the corresponding strengths and $\bxi(\omega) = (\xi_1(\omega),\xi_2(\omega), \ldots)^T$ are {independent and identically distributed (i.i.d.) uniform random variables} on $\left[ -\frac{1}{2}, \frac{1}{2} \right]$.
Thus,
$$
\bxi \in \left[ -\frac{1}{2}, \frac{1}{2} \right]^\bN =: U.
$$
According to the Doob-Dynkin lemma \cite{kallenberg1997foundations}, the solution $\psi$ can be represented by parametric functions parameterized by $\bxi$. Hence, (\ref{model origin}) can be rewritten as the following parametric nonlinear Schr\"{o}dinger equation with a random potential $V(\bxi,x)$ given by \eqref{random-potential origin}:
\begin{equation}\label{nls}
\left\{
\begin{aligned}
&i\partial_t\psi = -\frac{1}{2} \partial_x^2 \psi + V(\bxi,x) \psi + \alpha|\psi|^2\psi,
\quad x\in \bT, \ \bxi \in U,\ t>0,\\
&\psi(t=0,\bxi,x) = \psi_{in}(x),\quad x\in \bT,\ \bxi \in U,
\end{aligned}
\right.
\end{equation}
with $\psi=\psi(t,\bxi,x)$, $\psi_{in}$ being deterministic,  and
we shall consider that $\alpha \neq 0$.

Along the numerical aspect of (\ref{model origin}) or (\ref{nls}), let us mention several related works. The work of Henning and Peterseim \cite{dnls-fd} considered the fully deterministic case of (\ref{model origin}) and addressed the convergence issue of a Crank-Nicolson finite element discretization. Zhao later addressed the corresponding convergence of an exponential integrator spectral method in \cite{Zhao} and numerically investigated the stochastic case. The work of Kachman, Fishman and Soffer \cite{CAL2}  proposed an iterative integrator in time for the linear model. However, to our best knowledge, the existing studies on (\ref{model origin}) so far are yet to address the discretization in the random variable.  The sampling method in the aforementioned works as well as in the physical works  \cite{prl2,pra,Shepelyansky} all consider the classical Monte-Carlo method. It is known that the classical Monte-Carlo offers only a half-order convergence rate and so practical computing would require a large number of samples for accurately capturing the statistical quantities of interest. To increase the efficiency and accuracy of quantifying the uncertainty, methods like the quasi-Monte Carlo sampling \cite{Caflisch,Sloan Act,wang2003strong}, polynomial chaos expansion \cite{Cohen,Ghanem,HuJin,Xiu} and stochastic collocation method \cite{babuvska2007stochastic,nobile2008sparse,TangZhou} have later been developed for PDEs with random inputs.

In this work, we shall consider the quasi-Monte Carlo (QMC) sampling method for the NLS (\ref{nls}). The origin of QMC dates back to around 1960, which proposes to sample deterministically by using the low-discrepancy sequences \cite{halton1960efficiency,sobol1967distribution}.
{The convergence order of QMC with respect to the number of samples is higher than that of the classical Monte Carlo method; however, it is dependent on the dimension of the random variable \cite{Caflisch,niederreiter1992random}.} To overcome this problem, some state-of-art randomization techniques of different kinds have been further introduced to the QMC. We refer the interested readers to \cite{Sloan Act,kuo2011quasi} and the references therein for a detailed review. In this work, we would follow {the randomly shifted QMC lattice rule} \cite{Dick,Nuyens,Sloan_SINUM}. In such a way, Graham, Kuo, Schwab and Sloan et al. proposed the QMC method with finite element discretizations to solve linear elliptic equations \cite{gilbert2019analysis,Sloan NM,Sloan JCP,kuo2012quasi}. Via the technique of the weighted Hilbert space with the carefully chosen weights \cite{Sloan NM,kuo2012quasi}, they established the almost-linear convergence rates in $N$ without dependence on the dimension.

By combining the randomly shifted QMC lattice rule with the popular time-splitting Fourier pseudospectral (TSFP) method \cite{Besse,BaoCai,Jin}, we propose an efficient numerical method to solve (\ref{nls}). Moreover, we aim to analyze for the full convergence result of the scheme. The difficulty/novelty of the analysis mainly come from the following facts: i) the nonlinearity in the equation (\ref{nls}) makes it complicated to estimate of the regularity of the solution in the parametric space; ii) the physical observable of interest is a nonlinear functional of the solution; iii) the full scheme involves the discretizations in time, physical space and parametric space. Under suitable assumptions on the decaying rate of the potential (\ref{random-potential origin}), we point out the possible choice of the weights in the framework of the weighted Sobolev space, and we show the existence of a randomly shifted QMC lattice rule which can achieve an almost-linear convergence rate without dependence on dimensions for the expectation of the physical observable. The optimal root-mean-square error estimate of the QMC-TSFP scheme in time, physical space, and random space will then be established and verified numerically. Finally, we numerically investigate the {wave propagation} in the {random} NLS models.

The rest of the paper is organized as follows. In Section \ref{sec:method}, we present the detailed QMC-TSFP scheme and our main result about its convergence theorem. Section \ref{sec:pde} and Section \ref{sec:Fullscheme} are devoted to rigorously establishing the theorem by analyzing the error on the PDE and on the scheme in a sequel. Numerical results are presented in Section \ref{sec:result}.  Concluding remarks are made in Section \ref{sec:conclusion}.

\section{Numerical method}\label{sec:method}
In this section, we shall first present the numerical discretization of the NLS (\ref{nls}). Then, we shall give the  convergence result of the scheme.

\subsection{Dimension truncation and discretization}
First of all, for practical computation we need to truncate the Karhunen-Lo\`{e}ve expansion \eqref{random-potential origin} of the random potential $V(\bxi,x)$ into a finite sum. That is to say,
we choose an integer $m>0$ large enough and truncate the parameterized  random potential as:
\begin{align}
V_m(\bxi,x)=v_0(x)+\sum_{j=1}^{m}\sqrt{\lambda_j}\xi_j(\omega)v_j(x),\quad x\in \bT,\ {\xi = (\xi_1, \ldots, \xi_m)^T} \in U_m:=\left[-\frac12,\frac12\right]^m.
\label{random-potential}
\end{align}
We emphasize that the approximation of the truncated random potential $V_m(\bxi,x)$ to $V(\bxi,x)$ depends on the decaying rates of the strengths. With the above finite dimensional random potential, we consider the following truncated NLS problem:
\begin{equation}\label{nls trun}
\left\{
\begin{aligned}
&i\partial_t\psi_m = -\frac{1}{2} \partial_x^2 \psi_m + V_m(\xi,x) \psi_m + \alpha|\psi_m|^2\psi_m,
\quad x\in \bT, \ \xi \in U_m,\ t>0,\\
&\psi_m(t=0,\xi,x) = \psi_{in}(x),\quad x\in \bT,
\end{aligned}
\right.
\end{equation}
which is an approximation to (\ref{nls}).
Here we have $\psi_m(t,\xi,x)=\psi(t,(\bxi_{\{1:m\}},\mathbf{0}),x)$ with $\xi=\bxi_{\{1:m\}}$.

With one precise sample for $\xi$, (\ref{nls trun}) reads as a cubic NLS with the given potential $V_m$, and so it can be numerically solved by any of the classical algorithms. Concerning the periodic boundary condition, we shall consider in this paper the time-splitting Fourier pseudospectral method, which is simple to begin with but undoubtedly one of the most popular numerical methods for NLS \cite{BaoCai,Jin,Suzuki}. It begins by splitting (\ref{nls trun}) into two subflows $\Psi^\mathrm{k}_s$ and $\Psi^\mathrm{p}_s$ as
\begin{subequations}\label{subflow}
\begin{align}
\Psi^\mathrm{k}_s: i\partial_t\psi_m &= -\frac{1}{2} \partial_x^2 \psi_m,\quad t\in(0,s]; \\
\Psi^\mathrm{p}_s: i\partial_t\psi_m &= V_m\psi_m + \alpha|\psi_m|^2\psi_m,\quad t\in(0,s].
\end{align}
\end{subequations}
Note that $V_m$ and $\alpha$ are real-valued, so both of the above equations can be integrated exactly in time, and then the scheme composes as $\psi_m(t_{n+1},\xi,x)\approx \Psi_\frac{\tau}{2}^\mathrm{k}\circ \Psi_\tau^{\mathrm{p}}\circ
\Psi_\frac{\tau}{2}^\mathrm{k}(\psi_m(t_n,\xi,x))$, where $\tau=t_{n+1}-t_{n}$ is the time step. Let us briefly present the detailed scheme below.

Denote $\bT=[-L,L]$ with some $L>0$ and choose some even integer $M>0$. The physical space is discretized as $x_k=-L+kh$ with $h=2L/M$ for $k=0,1,\ldots,M$. For the time axis, we take $\tau=\Delta t>0$ and denote $t_n=n\tau$ for $n=0,1,\ldots$. Denote the numerical solution as $\psi_{m,k}^n=\psi_{m,k}^n(\xi)\approx \psi_m(t_n,\xi,x_k)$ with $\psi_{m,k}^0=\psi_{in}(x_k)$, the scheme of the \emph{time-splitting Fourier pseudospectral method (TSFP)} for solving (\ref{nls trun}) reads for $n\geq0$,
\begin{subequations}\label{scheme}
  \begin{align}
   &\psi_{m,k}^{n+1}=\sum_{l=-M/2}^{M/2-1}\widetilde{(\psi_{m,k}^{n+1})}_l\fe^{iu_l(x_k+L)},\quad
    \widetilde{(\psi_{m,k}^{n+1})}_l=\fe^{-iu_l^2\tau/4}\widetilde{(\psi_{m,k}^{**})}_l,\quad u_l=\frac{\pi l}{L},\\
    &\psi_{m,k}^{**}=\fe^{-i[V_m(\xi,x_k)+\alpha|\psi_{m,k}|^2]\tau}\psi_{m,k}^{*},\quad
    \widetilde{(\psi_{m,k}^{*})}_l=\fe^{-iu_l^2\tau/4}\widetilde{(\psi_{m,k}^{n})}_l,\quad 0\leq k\leq M,
  \end{align}
\end{subequations}
with $\widetilde{(\phi)}_l=\frac{1}{M}\sum_{k=0}^{M-1}\phi_k\fe^{-iu_l(x_k+L)}$ being the discrete Fourier coefficients of a discrete function $\phi_k$. With the discrete values, we can define a continuous version of the numerical solution as the interpolation:
\begin{equation}\label{num interpolate}
I_M\psi_m^n(x)=\sum_{l=-M/2}^{M/2-1}\widetilde{(\psi_{m,k}^{n})}_l\fe^{iu_l(x+L)},\quad
x\in \bT,\ n\geq0.\end{equation}
The TSFP (\ref{scheme}) is fully explicit in time and efficient with the computational cost $\mathcal{O}(M\log M)$ per time level under the help of fast Fourier transform. Moreover, it is time symmetric and preserves the discrete $l^2$-norm of the solution. These properties offer it unconditional stability and the reliable long-time performance \cite{Faou}. Thus, it has been used in \cite{soffer,Shepelyansky} to simulate the nonlinear Anderson models. {For simulating {wave propagation}, the torus domain setup for (\ref{nls}) or (\ref{nls trun}) is a valid truncation of the whole space problem provided that the domain  size $L>0$ is taken large enough and the initially localized solution in $\bT$ is away from the boundary within the time of computation.}

\subsection{Quasi-Monte Carlo sampling}
With the solution $\psi$ of (\ref{nls}), the physical observables in applications are some functionals of $\psi$, i.e., $G(\psi(t,\bxi,\cdot)),\  t\geq0, \ \bxi\in U,$
for some functional $G$ in the energy space of the solution.
Then, the expected value with respect to the random parameter is given by the following integral in infinite dimensions:
\begin{align}\label{expect def}
  \mathbb{E}[G](t):=\lim_{m\to\infty}\int_{U_m}G(\psi_m(t,\xi,\cdot))d\xi,
\end{align}
with the truncated solution $\psi_m$ of (\ref{nls trun}). Therefore, for some sufficiently large $m>0$, $\mathbb{E}[G]$ can be  approximated by
$\mathbb{E}[G](t)\approx\int_{U_m}G(\psi_m(t,\xi,\cdot))d\xi.$
{Note that $|\psi|^2$ represents the probability density function in quantum physics. Many widely concerned physical observables are linear functionals of $|\psi|^2$. One typical example is the center of mass, which is frequently considered in the context of Anderson localization \cite{soffer}.}
{Therefore, in the following, we shall focus on the expected value of $G(|\psi|^2)$, where $G$ is a linear functional. In order to estimate the convergence of QMC for computing the expected value of $G(|\psi|^2)$, we shall analyze the parametric regularity of $G(|\psi|^2)$ or equivalently the parametric regularity of $|\psi|^2$ in Section \ref{subsec: parametric regularity}, which is one of our main contributions in this work.} Let us denote
\begin{align}
&\mathcal{G}(t,\bxi) := G(|\psi(t, \bxi,\cdot)|^2), \qquad \mathcal{G}^m(t,\xi) := G(|\psi_m(t, \xi,\cdot)|^2),\label{G_k def}
\end{align}
for short in the following, and so we look for $\mathbb{E}[\mathcal{G}](t)$.

To evaluate the expectation of a general function $F(\xi)$ under a fixed $m$, the strategy of a quasi-Monte Carlo (QMC) method is an equal weight quadrature rule, i.e.,
\begin{align}\label{QMC1}
\mathbb{E}_m[F]:=\int_{U_m}F(\xi)d\xi\approx\frac{1}{N} \sum_{p = 1}^N {F}(\xi^{(p)})  =: Q_{m,N}[{F}],
\end{align}
where  $\xi^{(p)}\in U_m$ is the quadrature point and $N>0$ is the total number of samples.
Classical QMC is done by constructing some deterministic quadrature points, such as the Halton sequence \cite{halton1960efficiency} or the Sobol sequence \cite{sobol1967distribution}. However, a drawback is that the approximation error of (\ref{QMC1}) depends on the dimension $m$. When $m$ is large, the QMC might lose its efficiency over the standard Monte Carlo method.

To provide a practical error estimate and reduce the biased error,  some randomization techniques have been developed to construct the quadrature points. Here, we adopt the so-called \emph{shifted rank-one lattice rule} \cite{Sloan Act} to get $\xi^{(p)}$. That is
\begin{equation}\label{lattice}
  \xi^{(p)}=\mathrm{frac}\left(\frac{pz}{N}+\Delta\right)-\frac12,\quad p=1,\ldots,N,
\end{equation}
where  $\mathrm{frac}(w)$ means to take the fractional part of each component of a vector $w$, $\Delta\in[0,1]^m$ is a uniformly distributed random shift and $z\in\bZ^m$ is known as the generating vector. A specific generating vector $z$ can be constructed by the component-by-component construction (CBC) approach \cite{Sloan Act} to minimize the shifted-averaged worst-case error function, which will be detailed in our analysis later in Section \ref{sec:QMC-error}. {For every $\xi^{(p)}$, we solve the NLS \eqref{nls trun} using the TSFP scheme (\ref{scheme}). Then, we obtain $\{\psi_m^n(\xi^{(p)})\}_{p = 1}^N$ to compute the numerical expectation of some quantity of interest. The whole numerical scheme shall be referred to as QMC-TSFP in the rest of the paper.} The practical computing of $\mathbb{E}[\mathcal{G}]$ via the randomly shifted QMC-TSFP method is then implemented as:
\begin{enumerate}
\item Given the truncated dimension $m$ in \eqref{random-potential}, choose $M, \tau$ for discretizations in TSFP, and take $N$ as the number of samples for one random shift and $R$ as the number of random shifts;
\item Construct the generating vector $z$ using the CBC approach;
\item Generate i.i.d. random shifts $\Delta_1, \ldots, \Delta_R$ from the uniform distribution on $[0, 1]^m$. For each $\Delta_r$ with $r = 1, \ldots, R$, obtain the sample set $\{ \xi^{(r, p)} = \mathrm{frac}\left( \frac{pz}{N} + \Delta_r \right) - \frac{1}{2}: p = 1, \ldots, N \}$;
\item For each $\xi^{(r, p)}$ with $r = 1, \ldots, R$ and $p = 1, \ldots, N$, solve \eqref{nls trun} via \eqref{scheme} and obtain $I_M\psi_m^n(\xi^{r, p}, x)$ in (\ref{num interpolate});
\item Compute $\{ Q^{(r)}_{m, N}(G(|I_M\psi_m^n|^2)): r = 1, \ldots, R \}$, where
    \begin{align} \label{eq: QMC approximation using one single shift}
    Q^{(r)}_{m, N}(G(|I_M\psi_m^n|^2)) = \frac{1}{N} \sum_{p = 1}^N G(|I_M\psi_m^n(\xi^{(r, p)}, \cdot)|^2)
    \end{align}
    is the approximation of $\mathbb{E}_m[G(|I_M\psi_m^n|^2)]$ with one single shift $\Delta_r$ in the QMC rule;
\item Take the average over all the random shifts to get
    \begin{align} \label{eq: final QMC approximation}
    \overline{Q}_{m, N, R}(G(|I_M\psi_m^n|^2)) = \frac{1}{R} \sum_{r = 1}^R Q^{(r)}_{m, N}(G(|I_M\psi_m^n|^2))
    \end{align}
    as the final approximation of $\mathbb{E}[G(|\psi(t_n,\cdot,\cdot)|^2)]$.
\end{enumerate}
With a precise shift $\Delta_r$, the total computational cost of QMC-TSFP is $\mathcal{O}(NM\log M)$. The number of shifts $R$ is usually kept small in practice, e.g., $10$-$50$ \cite{Sloan Act}.

\subsection{Main result}\label{subsec:main}
We shall present here the main convergence result. {We first introduce the Bochner spaces we will consider for the potential $V$ and the wave functions $\psi$. For the potential $V$, we will consider the Bochner space $L^\infty(U;H^s(\bT))$ with the norm
$$
\| V \|_{L^\infty(U;H^s(\bT))} = \sup \limits_{\bxi \in U} \| V(\bxi) \|_{H^s(\bT)}.
$$
For the wave function $\psi$, we will consider the Bochner space $L^\infty((0,T) \times U;H^s(\bT))$ with the norm
$$
\| \psi \|_{L^\infty((0,T) \times U;H^s(\bT))} = \sup \limits_{t \in (0, T), \bxi \in U} \| \psi(t, \bxi) \|_{H^s(\bT)}.
$$}
Assume that the functional, the potential and the initial data of (\ref{nls}) satisfy the following conditions.
\begin{assumption}\label{assump}
Assume the linear functional $ G \in (H^{1}(\bT))' $, which is the dual space of $H^{1}(\bT)$. For some fixed $s\geq 1$, assume that $\psi_{in}\in H^s(\bT)$, $V\in L^\infty(U;H^s(\bT))$ and
$\|V_m\|_{L^\infty(U;H^s)}\leq \|V\|_{L^\infty(U;H^s)}$ for any $m>0$. {Concerning the relevant physical context, assume further that $\psi_{in}$ is localized in the torus domain $\bT=[-L,L]$ and $L>0$ is large enough such that the wave is yet to reach the boundary within the time of interest.}
\end{assumption}
In addition, to achieve the optimal convergence rate of QMC, we would further ask for the following technical assumption on $V$, {which shows the property of the decaying rates of the terms in the series (\ref{random-potential})}.
\begin{assumption}\label{assump2}
Assume that the potential $V$ in the form of  (\ref{random-potential origin}) satisfies:
\begin{equation}
\|V_m-V\|_{L^\infty(U;H^{1})}\leq Cm^{-\chi},\qquad
\sum_{j=1}^{\infty}\left[3^j\sqrt{\lambda_j}\|v_j\|_{H^{1}}\right]^{1/2}<\infty,
\end{equation}
for some constants $C,\chi>0$.
\end{assumption}
Then, we have the following error estimate result for the QMC-TSFP method.
\begin{theorem}\label{thm}
Assume that Assumption \ref{assump}, Assumption \ref{assump2} hold  and $\psi_{in}\in H^s(\bT)$ with $s\geq5$. Then, a randomly shifted QMC lattice rule can be constructed to solve (\ref{nls trun}) till any fixed $T>0$, and there exist constants $\tau_0,h_0>0$ independent of $m$ so that with the numerical solution $I_M\psi_m^n$ from the QMC-TSFP scheme (\ref{scheme}) for $\tau\leq \tau_0$ and $h\leq h_0$, the root-mean-square error satisfies: for any $\delta\in(0,1/2)$ and $N\leq10^{30}$,
\begin{align}\label{main estimate}
\sqrt{\mathbb{E}^{\Delta}\left[\Big|\mathbb{E}[\mathcal{G}](t_n)-Q_{m,N}[G(|I_M\psi_m^n|^2)]\Big|^2
\right]}\leq C
(\tau^2+h^{s-1}+m^{-\chi}+N^{\delta - 1}),
\end{align}
for all $t_n\in[0, T]$  and some constant $C>0$ independent of $m,\tau,h,N$, {where $\mathbb{E}^{\Delta}$ denotes the expectation with respect to the random shift.} 
\end{theorem}

The rest of the paper is devoted to rigorously establishing the above theorem. It will be done in a sequel by analyzing the QMC error on the PDE in Section \ref{sec:pde} and the error of the TSFP scheme in Section \ref{sec:Fullscheme}.

{
\begin{remark}
The condition $N\leq10^{30}$ is purely a technical condition since we use in the proof of Theorem \ref{thm} the property that the Euler totient function $\varphi(N)$ satisfies $1/\varphi(N)\leq 9/N$ for $N\leq10^{30}$ \cite{Sloan NM}; see also the proof of Lemma \ref{lm: qmc pde bound}. In practice, the number of samples will seldom exceed $10^{30}$.
\end{remark}
}

\begin{remark}
  The Assumption \ref{assump2} basically requires that the eigenpairs in (\ref{random-potential origin}) are exponentially decaying, which is stronger than the usual algebraic ones in the literature for linear model problems, e.g., \cite{Sloan NM,kuo2012quasi}. Such exponential decay assumption is a technical condition for the rigorous proof by means of analysis in the weighted Sobolev space. It is essentially raised by the nonlinearity. In practice, it could be satisfied if the random potential $V(\bxi,x)$ is a smooth and periodic (or Schwartz for the whole space case) function in terms of $x$ uniformly for all $\bxi$. Generally, the decaying speed of the eigenvalues depend on the regularity of the covarivance kernel. We refer interested readers to \cite{schwab2006karhunen} for more details.
\end{remark}

\begin{remark}
Under weaker assumptions than those in Theorem \ref{thm}, weaker convergence rates can be established similarly. For instance, if we only have $s\geq3$ then the temporal accuracy of TSFP will drop down to one (see \cite{Lubich} for the deterministic case). On the other hand, if the decaying rate  in Assumption \ref{assump2} is slower, then the sampling error from QMC might drop down to $1/2$ \cite{kuo2012quasi}. Here we omitted these cases for simplicity.
\end{remark}

{
\begin{remark}
The physical context of Anderson localization in NLS is associated with many physical parameters, including the time interval, spatial domain size, strength of nonlinearity, and randomness of the potential. Our primary focus in Theorem \ref{thm} is on the convergence rates with respect to numerical parameters. The dependence of the error estimate on physical parameters is encapsulated in the error constant $C>0$ in (\ref{main estimate}). Determining the explicit and particularly optimal dependence of the constant $C$ on these physical parameters requires substantial additional efforts and will be a subject of our future research.
\end{remark}
}

\section{QMC on the PDE}\label{sec:pde}

In this section, we aim to analyze the convergence of the QMC sampling for the NLS model (\ref{nls}) or (\ref{nls trun}) on the continuous level. We shall largely follow the framework in \cite{kuo2012quasi}.

\subsection{Regularity in physical space}

\begin{lemma}\label{lm1}
Under  Assumption \ref{assump}, for each $\bxi \in U$ the solution $\psi(t,\bxi,x)$ of (\ref{nls}) is globally well-posed and uniformly bounded in $H^1(\bT)$. Moreover, for any fixed $0<T<\infty$, we have $\psi\in L^\infty((0,T) \times U;H^s(\bT))$.
\end{lemma}
\begin{proof}
Following the framework from \cite{ginibre1978class}, let us briefly go through the proof to check mainly the effect from the random potential. For simplicity, we shall omit the spatial variable $x$ in the functions.
 By the Duhamel formula,  (\ref{nls}) reads
\begin{equation}\label{prop1:eq0}
\psi(t,\bxi) = \fe^{it\partial_x^2/2}\psi_{in}-i\int_0^t\fe^{i(t-\rho)\partial_x^2/2}\left[
V(\bxi)\psi(\rho,\bxi)+\alpha|\psi(\rho,\bxi)|^2\psi(\rho,\bxi)\right]d\rho,\quad t\geq0.
\end{equation}
Note that under condition $s\geq1$, we have the algebraic property of the Sobolev space $H^s$ in the 1D \cite[Chapter 4]{Adams}. Therefore, by taking the $H^s$-norm on both sides of {\eqref{prop1:eq0}}, we get
$$
\|\psi(t,\bxi)\|_{H^s}\leq\|\psi_{in}\|_{H^s}+C\int_0^t\left[\|V(\bxi)\|_{H^s}
\|\psi(\rho,\bxi)\|_{H^s}
+\|\psi(\rho,\bxi)\|_{H^s}^3\right]ds,
$$
for some constant $C>0$ from the Sobolev embedding which depends only on $s$. Then, {a bootstrap-type argument \cite{tao2006nonlinear} will give the local well-posedness of (\ref{nls}) in $H^s(\bT)$ for each $\bxi$ (see Appendix \ref{appendix: useful tools} for details of the bootstrap-type argument).} Moreover, there exists a $T_0>0$ such that $\psi\in L^\infty((0,T_0) \times U;H^s(\bT))$ owning to the fact that $\psi_{in}$ is deterministic and the assumption $V\in L^\infty(U;H^s(\bT))$.

On the other hand, since $V$ is real-valued, the model (\ref{nls}) enjoys the mass conservation
\begin{align} \label{eq: mass conservation}
M(t):=\int_{\bT}|\psi(t,\bxi,x)|^2dx\equiv \int_{\bT}|\psi_{in}(x)|^2dx=M(0),\quad t\geq0,
\end{align}
where the mass is independent of $\bxi$. Also, we have  the energy conservation:
\begin{align} \label{eq: energy conservation}
E(t,\bxi):=&\int_{\bT}\left[\frac12|\partial_x\psi(t,\bxi,x)|^2+V(\bxi,x) |\psi(t,\bxi,x)|^2+\frac{\alpha}{2}|\psi(t,\bxi,x)|^4 \right]dx \nonumber \\
\equiv& \int_{\bT}\left[\frac12|\partial_x\psi_{in}(x)|^2+V(\bxi,x)|\psi_{in}(x)|^2
+\frac{\alpha}{2}|\psi_{in}(x)|^4
\right]dx=E(0,\bxi),\quad  t\geq0,\ \bxi \in U.
\end{align}
Note by the Sobolev embedding that $H^s(\bT)\subseteq L^\infty(\bT)$, {then the condition $V\in L^\infty(U;H^s(\bT))$ guarantees a uniform upper bound of the initial energy for all the $\bxi$}, i.e.,
\begin{align} \label{eq: bound of energy}
|E(0,\bxi)|\leq \int_{\bT}\left[\frac12|\partial_x\psi_{in}(x)|^2+\|V\|_{L^\infty(U \times \bT)}|\psi_{in}(x)|^2
+\frac{|\alpha|}{2}|\psi_{in}(x)|^4
\right]dx=:E_0<\infty,\quad \bxi \in U.
\end{align}
{In the defocusing case, i.e., $\alpha>0$, it can be seen from the definition of $E(t,\bxi)$ in \eqref{eq: energy conservation} that
\begin{align} \label{eq: lower bound of energy in the defocusing case}
E(t,\bxi) & \ge \int_{\bT}\left[ \frac12|\partial_x\psi(t,\bxi,x)|^2+V(\bxi,x) |\psi(t,\bxi,x)|^2 \right]dx \nonumber \\
& \ge \frac{1}{2} \|\partial_x\psi(t,\bxi)\|_{L^2}^2 - \|V\|_{L^\infty(U \times \bT)} \| \psi(t, \bxi) \|_{L^2}^2,
\end{align}
which together with the mass conservation \eqref{eq: mass conservation}, the energy conservation \eqref{eq: energy conservation} and the bound of $E(0, \bxi)$ in \eqref{eq: bound of energy} shows that
\begin{align}
\|\partial_x\psi(t,\bxi)\|_{L^2}^2 & \le 2 |E(t,\bxi)| + 2\|V\|_{L^\infty(U \times \bT)} \| \psi(t, \bxi) \|_{L^2}^2 \nonumber \\
& \leq 2E_0+2\|V\|_{L^\infty(U \times \bT)} \| \psi_{in} \|_{L^2}^2,\quad \forall\, t\geq0,\,\bxi\in U.
\end{align}
In the focusing case, i.e., $\alpha<0$, we derive by the Gagliardo-Nirenberg inequality $\|\psi\|_{L^4}^4\leq C\|\partial_x\psi\|_{L^2}\|\psi\|_{L^2}^3$ and the Young's inequality that
\begin{align}
\|\psi(t,\bxi)\|_{L^4}^4\leq C\|\partial_x\psi(t,\bxi)\|_{L^2}\|\psi(t,\bxi)\|_{L^2}^3
\leq\frac{1}{-2\alpha}\|\partial_x\psi(t,\bxi)\|_{L^2}^2+\frac{-\alpha C^2}{2}\|\psi(t,\bxi)\|_{L^2}^6,
\end{align}
which together with the definition of $E(t,\bxi)$ in \eqref{eq: energy conservation} shows that
\begin{align} \label{eq: lower bound of energy in the focusing case}
E(t,\bxi) \ge \frac{1}{4} \|\partial_x\psi(t,\bxi)\|_{L^2}^2 - \|V\|_{L^\infty(U \times \bT)} \| \psi(t, \bxi) \|_{L^2}^2 - \frac{\alpha^2 C^2}{4} \| \psi(t, \bxi) \|_{L^2}^6.
\end{align}
Then the inequality \eqref{eq: lower bound of energy in the focusing case} together with the mass conservation \eqref{eq: mass conservation}, the energy conservation \eqref{eq: energy conservation} and the bound of $E(0, \bxi)$ in \eqref{eq: bound of energy} shows that
\begin{align}
\|\partial_x\psi(t,\bxi)\|_{L^2}^2 & \le 4|E(t, \bxi)| + 4 \|V\|_{L^\infty(U \times \bT)} \| \psi(t, \bxi) \|_{L^2}^2 + \alpha^2 C^2 \| \psi(t, \bxi) \|_{L^2}^6 \nonumber \\
& \leq 4E_0+4\|V\|_{L^\infty(U \times \bT)} \| \psi_{in} \|_{L^2}^2 + \alpha^2C^2 \| \psi_{in} \|_{L^2}^6,\quad
\forall\, t\geq0,\,\bxi\in U.
\end{align}}
Thus, $\|\psi(t,\bxi)\|_{H^1}$ for any $t\geq0,\,\bxi\in U$ is always bounded from above by a uniform constant depending only on  ${\| \psi_{in} \|_{L^2}^2}$ , $E_0$ and $\|V\|_{L^\infty(U \times \bT)}$. Such generic upper bound for $\|\psi(t,\bxi)\|_{H^1}$ together with the local well-posedness leads to the global well-posedness of (\ref{nls}) in $H^1(\bT)$ for each $\bxi$, and $\psi\in L^\infty(\bR^+ \times U;H^1(\bT))\subseteq L^\infty(\bR^+ \times U \times \bT)$.
Then, (\ref{prop1:eq0}) by {the product lemma \cite{tao2006nonlinear} (see Appendix \ref{appendix: useful tools} for details of the product lemma)} gives
$$
\|\psi(t,\bxi)\|_{H^s}\leq\|\psi_{in}\|_{H^s}+C\left[\|V\|_{L^\infty(U;H^s)}
+\|\psi\|_{L^\infty(\bR^+ \times U \times \bT)}^2\right]\int_0^t\|\psi(\rho,\bxi)\|_{H^s}d\rho,
$$
and the last assertion of the lemma follows from the Gronwall inequality for any fixed $T>0$.
\end{proof}

\begin{remark}\label{remark1}
For the truncated NLS problem (\ref{nls trun}), note that the truncation does not change the regularity of the random potential, as stated in Assumption \ref{assump}. Therefore, by the same procedure in Lemma \ref{lm1}, we can have $\psi_m\in L^\infty((0,T) \times U_m;H^s(\bT))$ for any finite $T>0$ with  $\|\psi_m\|_{L^\infty((0,T) \times U_m;H^s)}\leq \|\psi\|_{L^\infty((0,T) \times U;H^s)}$.
\end{remark}

\subsection{Error of dimension truncation}

\begin{lemma}\label{truncation error}
Under Assumption \ref{assump},  the difference between the solution of (\ref{nls trun}) and the solution of the parameterized approximation (\ref{nls}) is given by
  $$\|\psi_m-\psi\|_{L^\infty((0,T) \times U;H^{1})}\leq C\|V_m-V\|_{L^\infty(U;H^{1})},$$
  with a constant $C>0$ depending on $\alpha$, $T$, the norm of $V$ in $L^\infty(U;H^{s}(\bT))$ and the norm of $\psi$ in $L^\infty((0,T) \times U;H^{s}(\bT))$.
\end{lemma}
\begin{proof}
  Denoting $\delta\psi=\psi_m-\psi$ and taking the difference between  (\ref{nls trun}) and (\ref{nls}), we get
  \begin{equation*}
\left\{
\begin{aligned}
&i\partial_t\delta\psi=-\frac{1}{2}\partial_x^2\delta\psi+V
\delta\psi+(V_m-V)\psi_m+\alpha\left[|\psi_m|^2\psi_m-|\psi|^2\psi\right],\quad x\in \bT,\ t>0,\\
&\delta\psi(t=0)=0,\quad x\in \bT.
\end{aligned}
\right.
\end{equation*}
The Duhamel formula gives
\begin{align*} \delta\psi(t,\bxi)=-i\int_0^t\fe^{i(t-\rho)\partial_x^2/2}\Big[&V(\bxi)
\delta\psi(\rho,\bxi)+(V_m(\bxi)-V(\bxi))\psi_m(\rho,\bxi)\\
&+\alpha(|\psi_m(\rho,\bxi)|^2\psi_m(\rho,\bxi)-|\psi(\rho,\bxi)|^2\psi(\rho,\bxi))\Big]d\rho,\quad t\geq0,\end{align*}
which by the product lemma gives
\begin{align*}&\|\delta\psi(t,\bxi)\|_{H^{1}}\\
\leq &C\int_0^t\left[\|V\|_{L^\infty(U;H^s)}+\|\psi\|_{L^\infty((0,T) \times U;H^s)}^2
+\|\psi_m\|_{L^\infty((0,T) \times U_m;H^s)}^2\right]
\|\delta\psi(\rho,\bxi)\|_{H^{1}}d\rho\\
&+tC\|\psi_m\|_{L^\infty((0,T) \times U_m;H^s)}\|V_m-V\|_{L^\infty(U;H^{1})},\quad 0\leq t\leq T,
\end{align*}
with some $C>0$ depending on $\alpha$ only.  {Note that the function
$$
q(t) := tC\|\psi_m\|_{L^\infty((0,T) \times U_m;H^s)}\|V_m-V\|_{L^\infty(U;H^{1})}
$$
is non-decreasing in $t$ and the constant
$$
K := C \left[ \|V\|_{L^\infty(U;H^s)} + \|\psi\|_{L^\infty((0,T) \times U;H^s)}^2 + \|\psi_m\|_{L^\infty((0,T) \times U_m;H^s)}^2 \right] \ge 0.
$$
Then the Gronwall inequality leads to
\begin{align*}
\|\delta\psi(t,\bxi)\|_{H^{1}} \le & q(t) + K \exp(K t) \int_0^t \exp(- K \rho) q(\rho) d \rho \\
\le & q(t) \left( 1 + K \exp(K t) \int_0^t \exp(- K \rho) d \rho \right) = q(t) \exp(K t) \\
= & tC\|\psi_m\|_{L^\infty((0,T) \times U_m;H^s)}\|V_m-V\|_{L^\infty(U;H^{1})}
\\
&\times\exp{\left\{tC\left[\|V\|_{L^\infty(U;H^s)}+\|\psi\|_{L^\infty((0,T) \times U;H^s)}^2
+\|\psi_m\|_{L^\infty((0,T) \times U_m;H^s)}^2\right]\right\}},
\end{align*}
where we use the fact that $q(t)$ is non-decreasing in $t$ in the second inequality, and the uniform right-hand-side in $\bxi$ gives the assertion.}
\end{proof}
Under the condition $\|V_m-V\|_{L^\infty(U;H^{1})}\leq Cm^{-\chi}$ in Assumption \ref{assump2}, Lemma \ref{truncation error} tells that  $\psi_m\to\psi$ in $L^\infty((0,T) \times U;H^{1}(\bT))$ as $m\to\infty$.
Recall that with the given linear functional $G$, we are interested in computing $ \mathbb{E}[\mathcal{G}]$ which after the truncation is approximated by
\begin{equation*}
\mathbb{E}_m[\mathcal{G}^m]  = \int_{U_m} G(|\psi_m(t,\xi,\cdot)|^2) d \xi.
\end{equation*}
Its error is given as follows.

\begin{lemma}\label{lm:trun G}
Under Assumption \ref{assump} and Assumption \ref{assump2}, with $\mathcal{G}$ {and $\mathcal{G}^m$} defined in (\ref{G_k def}), we have the estimate
$$| \mathbb{E}[\mathcal{G}] - \mathbb{E}_m[\mathcal{G}^m]|\leq  Cm^{-\chi}, $$
{where $C$ is independent of $m$.}
\end{lemma}
\begin{proof}
By the definition (\ref{expect def}),
$\displaystyle \mathbb{E}[\mathcal{G}] - \mathbb{E}_m[\mathcal{G}^m]=\lim_{p\to\infty}\mathbb{E}_p[\mathcal{G}^p] - \mathbb{E}_m[\mathcal{G}^m]$. Note that when $p\geq m$, $\mathbb{E}_p[\mathcal{G}^m]=\mathbb{E}_m[\mathcal{G}^m]$, so we have
\begin{align}\label{lm:trun eq0}
 |\mathbb{E}[\mathcal{G}] - \mathbb{E}_m[\mathcal{G}^m]|=\lim_{p\to\infty}|\mathbb{E}_p[\mathcal{G}^p-\mathcal{G}^m]|.
\end{align}
By the assumptions and Lemma \ref{truncation error}, we have for all $t\in[0,T]$,
\begin{align}
|\mathcal{G}^p(t,\bxi_{\{1:p\}})-\mathcal{G}^m(t,\xi)|\leq
&\|G\|_{H^{1}(\bT)'}\||\psi_p(t,\bxi_{\{1:p\}},\cdot)|^2
-|\psi_m(t,\xi,\cdot)|^2\|_{H^{1}}.\label{lm2 eq2}
\end{align}
With the estimates in Lemma \ref{lm1} and Remark \ref{remark1}, we get by the triangle inequality and the algebraic property of $H^{1}$,
\begin{equation*}\||\psi_p(t,\bxi_{\{1:p\}},\cdot)|^2
-|\psi_m(t,\xi,\cdot)|^2\|_{H^{1}}\leq C\|\psi_p(t,\bxi_{\{1:p\}},\cdot)
-\psi_m(t,\xi,\cdot)\|_{H^{1}}.\end{equation*}
Plugging it into (\ref{lm2 eq2}) and then using triangle inequality and Lemma \ref{truncation error}, we find
\begin{align*}
|\mathcal{G}^p(t,\bxi_{\{1:p\}})-\mathcal{G}^m(t,\xi)|\leq&C\|\psi_p(t,\bxi_{\{1:p\}},\cdot)
-\psi(t,\bxi,\cdot)\|_{H^{1}}+C\| {\psi(t,\bxi,\cdot)}
-\psi_m(t,\xi,\cdot) \|_{H^{1}}
\\ \leq&C(\|V_p-V\|_{L^\infty(U;H^{1})}+\|V_m-V\|_{L^\infty(U;H^{1})})\leq C(m^{-\chi}+p^{-\chi}).
\end{align*}
Note here $C>0$ is uniform constant in $\bxi,p,m$. Then, we find
$$|\mathbb{E}_p[\mathcal{G}^p-\mathcal{G}^m]|\leq\mathbb{E}_p[|\mathcal{G}^p-\mathcal{G}^m|]\leq C(m^{-\chi}+p^{-\chi}),$$
and so by (\ref{lm:trun eq0}) we have the assertion of the lemma.
\end{proof}

\subsection{Regularity in parametric space}
\label{subsec: parametric regularity}

As one key to determine the convergence  of the QMC method, we need the regularity of the solution in the parametric space. We now start to estimate the mixed first derivatives of $\psi_m$ with respect to $\xi_j$. In the following, we shall denote $$\Gamma:=\{\nu=(\nu_1,\ldots,\nu_m)\,|\,
\nu_j=0\, \mbox{ or }\, 1,\, j=1,\ldots,m \}$$ with $|\nu|=\sum_{j=1}^{m}\nu_j$, and the mixed first derivative reads
$$\partial^{\nu}\psi_m=\partial_{\xi_1}^{\nu_1}
\ldots\partial_{\xi_m}^{\nu_m}\psi_m,\quad \nu\in\Gamma. $$

\begin{lemma}\label{lemma:regularitywrtrvs} Under Assumption \ref{assump},
for any fixed  $T>0,\ m>0$ and  multi-index $0\neq\nu\in \Gamma$, we have
\begin{equation}\label{lm3:res}
\|\partial^{\nu} \psi_m(t, \xi, \cdot)\|_{H^{1}}\leq 3^{(|\nu|+2)(|\nu|-1)/2}C_T^{2|\nu|-1}
\displaystyle\prod_{j = 1}^m \left(\sqrt{\lambda_j}\|v_j\|_{H^{1}}\right)^{\nu_j}, \  \forall\,\xi\in U_m,\ t\in[0,T],
\end{equation}
where $C_T=\max\{C_T^0,1\}$ with
\begin{align*}C_T^0:=&2\max\{1,|\alpha|\}C_0T\max\left\{1,\|\psi\|_{L^\infty((0,T)\times U;H^s)}\right\}\\
&\times
\exp\left(T \max\left\{1,C_0\left[\|V\|_{L^\infty(U;H^s)}+3|\alpha|\|\psi\|_{L^\infty((0,T) \times U;H^s)}^2\right]\right\}\right)
\end{align*}
for some generic constant $C_0>0$ independent of $\nu$ {and} $m$.
\end{lemma}
\begin{proof}
We prove (\ref{lm3:res}) by an induction on $|\nu|$.
  When $|\nu|=1$, we differentiate (\ref{nls}) with respect to $\xi_j$. By denoting
  $\partial_j\psi_m=\partial_{\xi_j}\psi_m$ and $\partial_jV_m=\partial_{\xi_j}V_m$ for short, we get for $j=1,\ldots,m$,
  $$i\partial_t\partial_j\psi_m=-\frac12\partial_x^2\partial_j\psi_m+\partial_jV_m\psi_m+
  V_m\partial_j\psi_m+\alpha\left(2|\psi_m|^2\partial_j\psi_m+\psi_m^2\overline{\partial_j\psi_m}
  \right).$$
Omitting the spatial variable for brevity, the Duhamel formula of the above shows
\begin{align*}
\partial_j\psi_m(t, \xi) = &\fe^{it\partial_x^2/2}\partial_j\psi_{in}
- i\int_0^t\fe^{i(t-\rho)\partial_x^2/2}\Big[\partial_jV_m(\xi)\psi_m(\rho,\xi)+
V_m(\xi)\partial_j\psi_m(\rho,\xi)\\
&+\alpha\left(2|\psi_m(\rho,\xi)|^2\partial_j\psi_m(\rho,\xi)+\psi_m(\rho,\xi)^2
\overline{\partial_j\psi_m}(\rho,\xi)
\right)\Big]d\rho.
\end{align*}
Since the initial function is deterministic, so $\partial_j\psi_{in}=0$. Then we find
\begin{align*}
\|\partial_j\psi_m(t,\xi)\|_{H^{1}}\leq C_0\int_0^t\Big[&\|\partial_jV_m(\xi)\|_{H^{1}}\|\psi_m(\rho,\xi)\|_{H^{1}}
+\|V_m(\xi)\|_{H^{1}}\|\partial_j\psi_m(\rho,\xi)\|_{H^{1}}\\
&+3|\alpha|\|\psi_m(\rho,\xi)\|_{H^{1}}^2\|\partial_j\psi_m(\rho,\xi)\|_{H^{1}}\Big]d\rho,
\end{align*}
with a generic constant $C_0>0$ coming from the Sobolev embedding.
By the fact $\partial_jV_m=\sqrt{\lambda_j}v_j$, and the boundedness
$$\|\psi_m(t,\xi)\|_{H^{1}}\leq
\|\psi_m\|_{L^\infty((0,T) \times U_m;H^{1})}\leq\|\psi\|_{L^\infty((0,T) \times U;H^{s})} ,\quad t\in[0,T],$$ as stated in Remark \ref{remark1}, we further find for $t\in[0,T]$,
\begin{align*}
\|\partial_j\psi_m(t,\xi)\|_{H^{1}}\leq &C_0t\sqrt{\lambda_j}\|v_j\|_{H^{1}}\|\psi\|_{L^\infty((0,T) \times U;H^s)}\\
&+ C_0\left[\|V\|_{L^\infty(U;H^s)}+3|\alpha|\|\psi\|_{L^\infty((0,T) \times U;H^s)}^2\right]\int_0^t
\|\partial_j\psi_m(\rho,\xi)\|_{H^{1}}d\rho.
\end{align*}
By Gronwall's inequality, we get for $t\in[0,T],\ \xi \in U_m$,
\begin{align*}
 &\|\partial_j\psi_m(t,\xi)\|_{H^{1}}\\
 \leq &\exp\left(tC_0\left[\|V\|_{L^\infty(U;H^s)}+3|\alpha|\|\psi\|_{L^\infty((0,T) \times U;H^s)}^2\right]\right)C_0t\|\psi\|_{L^\infty((0,T) \times U;H^s)}
\sqrt{\lambda_j}\|v_j\|_{H^{1}},
\end{align*}
which verifies (\ref{lm3:res}) for $|\nu|=1$.

When $|\nu|\geq2$, by the Leibniz rule we have
\begin{align*}
i\partial_t\partial^{\nu}\psi_m=&-\frac{\partial_x^2}{2}\partial^{\nu}\psi_m
+\sum_{\mu\preceq\nu}\binom{{\nu}}{\mu}\partial^{{\nu}-{\mu}}V_m\partial^{\mu}\psi_m
+\alpha\sum_{{\mu}\preceq{\nu}}\binom{{\nu}}{{\mu}}
\partial^{{\nu}-{\mu}}|\psi_m|^2\partial^{\mu}\psi_m\\
=&-\frac{\partial_x^2}{2}\partial^{\nu}\psi_m+V_m\partial^{\nu}\psi_m
+\alpha\left(2|\psi_m|^2\partial^\nu\psi_m+\psi_m^2\partial^\nu\overline{\psi_m}\right)
+\sum_{|{\nu}-{\mu}|=1,\mu\prec\nu}\partial^{{\nu}-{\mu}}V_m\partial^{\mu}\psi_m\\
&+\alpha\sum_{0\neq{\mu}\prec{\nu}}\sum_{\eta\preceq\nu-\mu}
\partial^{{\nu}-{\mu}-\eta}\psi_m\partial^{\mu}\psi_m\partial^{\eta}\overline{\psi_m}.
\end{align*}
Here $\binom{\nu}{\mu}=\prod_{j=1}^{m}\binom{\nu_j}{\mu_j}$.
Similarly as before, we can deduce
\begin{align*}
&\|\partial^{\nu}\psi_m(t,\xi)\|_{H^{1}}\\
 \leq &C_0\int_0^t\left[\|V\|_{L^\infty(U;H^s)}+  3|\alpha|
 \|\psi\|_{L^\infty((0,T) \times U;H^s)}^2\right]\|\partial^{\nu}\psi_m(\rho,\xi)\|_{H^{1}}d\rho\\
 &+C_0\int_0^t \sum_{|{\nu}-{\mu}|=1,\mu\prec\nu}
 \|\partial^{{\nu}-{\mu}}V_m\|_{H^{1}}\|\partial^{\mu}\psi_m(\rho,\xi)\|_{H^{1}}d\rho\\
 &+C_0|\alpha|\int_0^t \sum_{0\neq{\mu}\prec{\nu}}\sum_{\eta\preceq\nu-\mu}
 \|\partial^{{\nu}-{\mu}-\eta}\psi_m(\rho,\xi)\|_{H^{1}}\|\partial^{\mu}\psi_m(\rho,\xi)\|_{H^{1}}
 \|\partial^{\eta}\psi_m(\rho,\xi)\|_{H^{1}}d\rho.
\end{align*}
By the Gronwall inequality, we have
\begin{align}
\|\partial^{\nu}\psi_m(t,\xi)\|_{H^{1}}
 \leq & \tilde{C}\bigg[\sum_{|{\nu}-{\mu}|=1,\mu\prec\nu}
 \|\partial^{{\nu}-{\mu}}V_m\|_{H^{1}}\int_0^T\|\partial^{\mu}\psi_m(\rho,\xi)\|_{H^{1}}d\rho\label{lm3:eq0}\\
 &+\sum_{0\neq{\mu}\prec{\nu}}\sum_{\eta\preceq\nu-\mu}\int_0^T
 \|\partial^{{\nu}-{\mu}-\eta}\psi_m(\rho,\xi)\|_{H^{1}}\|\partial^{\mu}\psi_m(\rho,\xi)\|_{H^{1}}
 \|\partial^{\eta}\psi_m(\rho,\xi)\|_{H^{1}}d\rho\bigg],\nonumber
\end{align}
with $$\tilde{C}=\max\{1,|\alpha|\}
C_0\exp\left(T\max\left\{1,C_0\left[\|V\|_{L^\infty(U;H^s)}+3|\alpha|\|\psi\|_{L^\infty((0,T) \times U;H^s)}^2\right]\right\}\right).$$
For the first summation term above, noting that $|\mu|=|\nu|-1$ and using an induction argument with the assumption (\ref{lm3:res}), we have
\begin{align}\label{lm3: eq1}
\sum_{|{\nu}-{\mu}|=1,\mu\prec\nu}
 \|\partial^{{\nu}-{\mu}}V_m\|_{H^{1}}\|\partial^{\mu}\psi_m(\rho,\xi)\|_{H^{1}}\leq |\nu|
 3^{(|\nu|+1)(|\nu|-2)/2}C_T^{2|\nu|-3}
\displaystyle\prod_{j = 1}^m \left(\sqrt{\lambda_j}\|v_j\|_{H^{1}}\right)^{\nu_j}.
 \end{align}
On the other hand, for the double summation part in (\ref{lm3:eq0}), by the induction assumption (\ref{lm3:res}) we find
\begin{align*}
 &\sum_{0\neq{\mu}\prec{\nu}}\sum_{\eta\preceq\nu-\mu}
 \|\partial^{{\nu}-{\mu}-\eta}\psi_m(\rho,\xi)\|_{H^{1}}\|\partial^{\mu}\psi_m(\rho,\xi)\|_{H^{1}}
 \|\partial^{\eta}\psi_m(\rho,\xi)\|_{H^{1}}\\
 \leq& \sum_{0\neq{\mu}\prec{\nu}}\sum_{\eta\preceq\nu-\mu}C_T^{2|\nu|-2}\displaystyle\prod_{j = 1}^m \left(\sqrt{\lambda_j}\|v_j\|_{H^{1}}\right)^{\nu_j}
 3^{\frac{|\nu|^2+|\nu|-6}{2}+|\mu|^2+|\eta|^2-|\nu||\mu|-|\nu||\eta|-|\mu||\eta|},
\end{align*}
and it is direct to check here that $$\frac{|\nu|^2+|\nu|-6}{2}+|\mu|^2+|\eta|^2-|\nu||\mu|-|\nu||\eta|-|\mu||\eta|
\leq \frac{|\nu|^2-|\nu|-2}{2}.$$ So by
noting that the double summation contains $3^{|\nu|}-3$ terms in total, we further get
 \begin{align}
 &\sum_{0\neq{\mu}\prec{\nu}}\sum_{\eta\prec\nu-\mu}
 \|\partial^{{\nu}-{\mu}-\eta}\psi_m(\rho,\xi)\|_{H^{1}}\|\partial^{\mu}\psi_m(\rho,\xi)\|_{H^{1}}
 \|\partial^{\eta}\psi_m(\rho,\xi)\|_{H^{1}}\nonumber\\
\leq& 3^{|\nu|}C_T^{2|\nu|-2}\displaystyle\prod_{j = 1}^m \left(\sqrt{\lambda_j}\|v_j\|_{H^{1}}\right)^{\nu_j}
 3^{(|\nu|+1)(|\nu|-2)/2}\label{lm3:eq2}.
\end{align}
 By plugging (\ref{lm3:eq2}) and (\ref{lm3: eq1}) into (\ref{lm3:eq0}), we obtain
 \begin{align*}
\|\partial^{\nu}\psi_m(t,\xi)\|_{H^{1}}\leq &
\tilde{C}TC_T^{2|\nu|-2}3^{(|\nu|+1)(|\nu|-2)/2}
 \left(|\nu|+3^{|\nu|}\right)
 \displaystyle\prod_{j = 1}^m \left(\sqrt{\lambda_j}\|v_j\|_{H^{1}}\right)^{\nu_j}.
\end{align*}
Further noting that $\tilde{C}T\leq C_T^0/2$ and $|\nu|<3^{|\nu|}$, we find
 \begin{align*}
\|\partial^{\nu}\psi_m(t,\xi)\|_{H^{1}}
 \leq &\frac{1}{2}C_T^{2|\nu|-1}3^{(|\nu|+1)(|\nu|-2)/2}
 \left(|\nu|+3^{|\nu|}\right)
 \displaystyle\prod_{j = 1}^m \left(\sqrt{\lambda_j}\|v_j\|_{H^{1}}\right)^{\nu_j}\\
 \leq& C_T^{2|\nu|-1}3^{(|\nu|+2)(|\nu|-1)/2}
 \displaystyle\prod_{j = 1}^m \left(\sqrt{\lambda_j}\|v_j\|_{H^{1}}\right)^{\nu_j}.
\end{align*}

\end{proof}

%



With the established regularity of the solution in the parametric space, we can now estimate the mixed derivative of the physical observable $\mathcal{G}^m(t,\xi)=G(|\psi_m(t,\xi,\cdot)|^2)$ with respect to $\xi$.

\begin{lemma} \label{corF}
Under Assumption \ref{assump}, for any multi-index ${\nu} \in \Gamma$ with the constant $
C_\nu:=2^{|\nu|}C_T^{2|\nu|-1}3^{(|\nu|+2)(|\nu|-1)/2}$,
we have
\begin{equation}
\vert \partial^{\nu} \mathcal{G}^m (t,\xi)  \vert \leq C_0C_\nu \prod_{j = 1}^m \left(\sqrt{\lambda_j}\|v_j\|_{H^{1}}\right)^{\nu_j}  \Vert G \Vert_{H^{1}(\bT)'}, \quad \forall\,\xi \in U_m,\ t\in[0,T].
\label{MixedDerivativeMainEstmate}
\end{equation}
\end{lemma}

\begin{proof}
For any $\xi\in U_m$ and $t\in[0,T]$, we have
\begin{align*}
\vert \partial^{\nu} \mathcal{G}^m(t,\xi)  \vert &=\vert G \left( \partial^{\nu} |\psi_m|^2 \right) \vert \le \Vert G \Vert_{H^{1}(\bT)'} \| \partial^{\nu} |\psi_m|^2 \|_{H^{1}(\bT)}.
\end{align*}
By the algebraic property of $H^{1}$ and Lemma \ref{lemma:regularitywrtrvs}, we find
\begin{align*}
\vert \partial^{\nu} \mathcal{G}^m(t,\xi)  \vert &\leq\Vert G \Vert_{H^{1}(\bT)'} \sum_{{\mu}\preceq{\nu}}C_0
\left\|\partial^{{\nu}-{\mu}}\overline{\psi_m}\right\|_{H^{1}}\|\partial^{\mu}\psi_m\|_{H^{1}}\\
&\leq C_0C_\nu\Vert G \Vert_{H^{1}(\bT)'}
\displaystyle\prod_{j = 1}^m \left(\sqrt{\lambda_j}\|v_j\|_{H^{1}}\right)^{\nu_j}.
\end{align*}
\end{proof}

\subsection{Analysis of the QMC integration error}\label{sec:QMC-error}

With the preparation before, we now specify the QMC method, i.e., the quadrature points in (\ref{lattice}),  and then we shall derive its  error bound on the PDE level.

We adopt the weighted Sobolev space technique \cite{kuo2012quasi,Sloan_SINUM} here.
Consider the weighted and unanchored Sobolev space $\mathcal{W}_{m,\gamma}:=\{F(\xi):U_m\to\bR\  |\ \|F\|_{W_{m,\gamma}}<\infty\}$ with the norm
\begin{equation}\label{norm def}
\|F\|_{W_{m,\gamma}}:=\left[\sum_{{ \nu\in\Gamma}}\normalsize\gamma_\nu^{-1}
\int_{U_{|\nu|}}\left(\int_{U_{m-|\nu|}}\partial^\nu F(\xi)d\xi_{\nu}^\dag\right)^2d\xi_{\nu}\right]^{1/2},\end{equation}
where by the notation, $\xi$ is split into the active part $\xi_{\nu}$ for differentiations and the inactive part $\xi_{\nu}^\dag$, {i.e., $\xi_{\nu}$ consists of $\xi_j$ with $j$ such that $\nu_j = 1$ and $\xi_{\nu}^\dag$ consists of $\xi_k$ with $k$ such that $\nu_k = 0$}. The notation $\gamma_\nu$ denotes the product-type weight, i.e.,
$\gamma_\nu=\Pi_{\nu_j=1}\gamma_j$ with $\gamma_{\mathbf{0}_{\{1:m\}}}=1$,
which will be  chosen later, and   $\gamma:=\{\gamma_{\nu}\,|\, \nu\in\Gamma\}$. With the chosen $\gamma_j$ ($j=1,\ldots,m$) and the following worst-case error function
$$e^{wt}_{s}(z):=-1+\frac1N\sum_{p=1}^{N}\prod_{j=1}^s\left[1+\gamma_jB\left(\mathrm{frac}(pz_j/N)\right)
\right],\quad z=(z_1,\ldots,z_s),\ 1\leq s\leq m,$$
where $B(x)=-1/(2\pi^2)\sum_{l\in\bZ\setminus\{0\}}\fe^{2\pi ilx}/l^2$ is the Bernoulli polynomial,  the generating vector $z$ in (\ref{lattice}) is computed by the CBC algorithm \cite{Sloan Act}. That is to set $z_1=1$, and then determine each $z_s$ for $s=2,\ldots,m$ in a sequel by minimizing $e^{wt}_{s}(z)$ for $z_s\in\mathbb{U}_N=\{x\in\bZ\,|\, 1\leq x\leq N-1,\, \mathrm{gcd}(x,N)=1\}$. An explicit error bound of the QMC with  such constructed quadrature points can be established (see e.g., \cite{kuo2011quasi}). Here let us quote it as the following lemma.

\begin{lemma}(\cite[{Theorem 4.1}]{kuo2011quasi})
  With fixed $m,\,N,\,\gamma$ and some $F(\xi)\in W_{m,\gamma}$, the QMC method (\ref{lattice}) given by the CBC algorithm satisfies
  \begin{align}\label{worst error bound}
  \sqrt{\mathbb{E}^\Delta\left[|E_m(F)-Q_{m,N}[F]|^2\right]}
  \leq \left[\sum_{\nu\in\Gamma\setminus\mathbf{0}_{\{1:m\}}}\gamma_\nu^\lambda
  \left(\frac{2\zeta(2\lambda)}{(2\pi^2)^\lambda}\right)^{|\nu|}\right]^{1/(2\lambda)}
  \varphi(N)^{-1/(2\lambda)}\|F\|_{W_{m,\gamma}},
  \end{align}
  for any $\lambda\in(\frac12,1]$. Here $\varphi(N)=|\mathbb{U}_N|$ is the Euler totient function and $\zeta(\cdot)$ is the Riemann zeta function.
\end{lemma}

The possible choice of the weight $\gamma$ to provide the optimal convergence rate for the QMC is specified in the following lemma.

\begin{lemma}\label{lm: qmc pde bound}
 With fixed $m\in\bN_+$ and $N\leq10^{30}$, 
 choose $\lambda=1/(2-2\delta)$ for any $\delta\in(0,1/2)$.  
Then under Assumption \ref{assump} and Assumption \ref{assump2}, a QMC method (\ref{lattice}) can be constructed by the CBC algorithm which satisfies
  \begin{align}\label{error bound m}
  \sqrt{\mathbb{E}^\Delta\left[|E_m[\mathcal{G}^m](t)-Q_{m,N}[\mathcal{G}^m](t)|^2\right]}
  \leq CN^{\delta - 1},\quad t\in[0,T],
  \end{align}
where $C>0$ is some constant independent of $m$.
\end{lemma}
\begin{proof}
It is direct to compute from the definition of the weighted norm (\ref{norm def}) and Lemma \ref{corF} that
\begin{align*}\|\mathcal{G}^m(t,\cdot)\|_{W_{m,\gamma}}=&\left[\sum_{{ \nu\in\Gamma}}\normalsize\gamma_\nu^{-1}
\int_{U_{|\nu|}}\left(\int_{U_{m-|\nu|}}\partial^\nu \mathcal{G}^m(t,\xi)d\xi_{\nu}^\dag\right)^2d\xi_{\nu}\right]^{1/2}\\
\leq &C_0\|G\|_{H^{1}(D)'}\left[\sum_{\nu\in\Gamma}\gamma_{\nu}^{-1}C_\nu^2\prod_{j=1}^{m}\lambda_j^{\nu_j}
\|v_j\|_{H^{1}}^{2\nu_j}\right]^{1/2}.
\end{align*}
Meanwhile, it is known that the Euler totient function satisfies $1/\varphi(N)\leq 9/N$ for $N\leq10^{30}$ \cite{Sloan NM}. Thus, by \eqref{worst error bound} and choosing  $\lambda=1/(2-2\delta)$ we find
\begin{align*}
  \sqrt{\mathbb{E}^\Delta\left[|E_m[\mathcal{G}^m](t)-Q_{m,N}[\mathcal{G}^m](t)|^2\right]}
  \leq 9C_\gamma C_0  \|G\|_{H^{1}(\bT)'}N^{\delta-1},
\end{align*}
with
$$C_\gamma=\left[\sum_{\nu\in\Gamma}\gamma_{\nu}^{-1}C_\nu^2\prod_{j=1}^{m}\lambda_j^{\nu_j}
\|v_j\|_{H^{1}}^{2\nu_j}\right]^{1/2}\left[\sum_{\nu\in\Gamma}\gamma_\nu^\lambda
  \left(\frac{2\zeta(2\lambda)}{(2\pi^2)^\lambda}\right)^{|\nu|}\right]^{1/(2\lambda)}.$$
  Denote $b_j=\sqrt{\lambda_j}\|v_j\|_{H^{1}}$ and $\rho(\lambda)=2\zeta(2\lambda)/(2\pi^2)^\lambda$ for simplicity.
Now we take
$$\gamma_\nu=\left[C_\nu\prod_{j=1}^m\frac{b_j^{\nu_j}}{\rho(\lambda)^{\nu_j/2}}\right]^{2/(1+\lambda)},$$
and it is direct to deduce that
$C_\gamma=(A_\lambda)^{\frac{\lambda+1}{2\lambda}}$ with
$$A_\lambda=\sum_{\nu\in\Gamma} C_\nu^{2\lambda/(1+\lambda)}\prod_{j=1}^{m}\left[b_j^{2\lambda\nu_j}
\rho(\lambda)^{\nu_j}\right]^{1/(1+\lambda)}.$$
It remains to show that $A_\lambda$ is bounded $\forall \nu\in\Gamma,m\in\bN$.

By the constant $C_\nu$ from Lemma \ref{corF} and using H\"older's inequality, 
we find
\begin{align*}
 A_\lambda\leq&\sum_{\nu\in\Gamma}\left[3^{|\nu|^2+|\nu|-1/2}\prod_{j=1}^{m}
 \alpha_j^{\nu_j}\right]^{2\lambda/(1+\lambda)}
 \prod_{j=1}^{m}\left[\frac{(2C_T^2)^{2\lambda}b_j^{2\lambda}
\rho(\lambda)}{\alpha_j^{2\lambda}}\right]^{\nu_j/(1+\lambda)}\\
\leq&\left(\sum_{\nu\in\Gamma}\prod_{j=1}^{m}
 3^{\nu_js_j}\alpha_j^{\nu_j}\right)^{2\lambda/(1+\lambda)}\left(\sum_{\nu\in\Gamma}
 \prod_{j=1}^{m}\left[\frac{(2C_T^2)^{2\lambda}b_j^{2\lambda}
\rho(\lambda)}{\alpha_j^{2\lambda}}\right]^{\nu_j/(1-\lambda)}\right)^{(1-\lambda)/(1+\lambda)},
\end{align*}
where $s_j:=\sum_{l=1}^j\nu_l$ and $\alpha_j>0$ is chosen as follows. Note the fact that
$\sum_{\nu\in\Gamma}\prod_{j=1}^{m}\beta_j^{\nu_j}\leq \exp(\sum_{j\geq1}\beta_j)$ holds for any $m\geq1$ and  any $\beta_j>0$ with $\sum_{j\geq1}\beta_j<\infty$ (see \cite{kuo2012quasi}), then we have: $\forall \nu\in\Gamma,m\in\bN$,
\begin{align*}
 A_\lambda
\leq&\exp\left(\frac{2\lambda}{1+\lambda}\sum_{j\geq1}
 3^{j}\alpha_j\right)\exp\left(\frac{1-\lambda}{1+\lambda}\sum_{j\geq1}
 \left[\frac{(2C_T^2)^{2\lambda}b_j^{2\lambda}
\rho(\lambda)}{\alpha_j^{2\lambda}}\right]^{1/(1-\lambda)}\right)=:A_{max} .
\end{align*}
By taking $\alpha_j=\sqrt{b_j/3^j}$ and under  Assumption \ref{assump2}, we have $\sum_{j\geq1}
 3^{j}\alpha_j<1$. Moreover, with $\lambda\in(\frac12,1]$ we have
 $\frac{\lambda}{1-\lambda}\geq 1$ and so
$$\sum_{j\geq1}(b_j/\alpha_j)^{2\lambda/(1-\lambda)}=\sum_{j\geq1}
(3^j b_j)^{\lambda/(1-\lambda)}<\infty.$$
Thus, we find the upper bound $A_{max}$ is finite.
\end{proof}

\begin{proposition}\label{prop1}
Under the same assumption of Lemma \ref{lm: qmc pde bound},  there exists a constant $C>0$ independent of $m$  such that
\begin{align}
\sqrt{\mathbb{E}^{\Delta}\left[\left|\mathbb{E}[\mathcal{G}](t) - Q_{m,N}[\mathcal{G}^m](t) \right|^2\right]} \ \leq C(m^{-\chi}+ N^{\delta - 1}),\quad t\in[0,T],
\label{QMC-error}
\end{align}
for any $\delta\in(0,1/2)$.
\end{proposition}
\begin{proof}
Firstly, by triangle inequality we have
\begin{align*}
|\mathbb{E}[\mathcal{G}](t) - Q_{m,N}[\mathcal{G}^m](t) |\leq | \mathbb{E}[\mathcal{G}](t) - \mathbb{E}_m[\mathcal{G}^m]| +| \mathbb{E}_m[\mathcal{G}^m] - Q_{m,N}[\mathcal{G}^m](t) |.
\end{align*}
The first part of right-hand-side is independent of the random shift $\Delta$, so
\begin{align*}\mathbb{E}^\Delta\left[|\mathbb{E}[\mathcal{G}](t) - Q_{m,N}[\mathcal{G}^m](t) |^2\right]\leq& 2\left|\mathbb{E}[\mathcal{G}](t) - \mathbb{E}_m[\mathcal{G}^m](t)\right|^2 \\
&+2\mathbb{E}^\Delta\left[| \mathbb{E}_m[\mathcal{G}^m](t) - Q_{m,N}[\mathcal{G}^m](t) |^2\right].
\end{align*}
Then by Lemma \ref{lm:trun G} and Lemma \ref{lm: qmc pde bound}, we obtain the result.
\end{proof}

\begin{remark} \label{remark: exsistence of QMC that converges almost linearly}
We remark the choice of the weight $\gamma$ and the error constant given in our analysis above may be far away from optimal. Our result theoretically serves to show the existence of the randomly shifted QMC lattice rule that can converge at the desired {almost-linear} and dimension-independent rate for solving the  {random} NLS (\ref{nls}).
\end{remark}

\section{Error of the scheme}\label{sec:Fullscheme}
In this section, we continue the analysis to complete the full error estimate result  given in Section \ref{subsec:main}. To do so, we first carry out the error estimate of the TSFP scheme on the truncated problem (\ref{nls trun}), which is stated as the following.

\begin{proposition}\label{prop2}
  Let $\psi_m^n$ be the numerical solution from  the TSFP scheme (\ref{scheme}) for solving the truncated problem (\ref{nls trun}) till any fixed $T>0$ with any $m\in\bN$. Under the same assumption in Theorem \ref{thm}, there exist constants $\tau_0,h_0>0$ independent of $m$ such that when $\tau\leq\tau_0$ and $h\leq h_0$, the following error bound holds
  $$\|I_M\psi_m^n-\psi_m(t_n,\xi,\cdot)\|_{H^{1}}\leq C(\tau^2+h^{s-1}),\quad \forall\, 0\leq t_n\leq T,\
  \xi\in U_m,$$
  for some constant $C>0$ independent of $\xi,m,\tau$ and $h$.
\end{proposition}

\begin{proof}
The proof largely follows the framework of analysis for the Strang splitting in \cite{BaoCai,Lubich}. Here we shall go through it mainly to check the impact from the parametric space. Let us omit the variables $x,\xi$ and denote $\psi_m(t)=\psi_m(t,\xi,x)$ for simplicity.

We begin with the estimate of the temporal truncation error. 
Denoting $v^n(\rho):=\fe^{\frac12 i\partial_x^2\rho}\psi_m(t_n)$ for short,
by Taylor's expansion we find
\begin{align*}&\Psi_\tau^{\mathrm{p}}\circ
\Psi_\frac{\tau}{2}^\mathrm{k}(\psi_m(t_n))
=\left[1+i\tau(V_m+\alpha|v^n(\tau/2)|^2)-
\frac{\tau^2}{2}(V_m+\alpha|v^n(\tau/2)|^2)^2\right]v^n(\tau/2)
 +r_1^n,
\end{align*}
where $r_1^n$ depends on $V_m$ and $v^n$. By Remark \ref{remark1}, we have $\|r_1^n\|_{H^1}\leq C\tau^3$ with $C$ depending on the $H^1$-norm of $V_m,\psi_m$.
Then,
\begin{align}\label{split err eq1}
 \Psi_\frac{\tau}{2}^\mathrm{k}\circ \Psi_\tau^{\mathrm{p}}\circ
\Psi_\frac{\tau}{2}^\mathrm{k}(\psi_m(t_n))=&v^n(\tau)+
i\tau\fe^{\frac12 i\partial_x^2\tau/2}\left(V_m+\alpha|v^n(\tau/2)|^2\right)v^n(\tau/2)\nonumber\\
&-\frac{\tau^2}{2}\left(V_m+\alpha|v^n(\tau/2)|^2\right)^2v^n(\tau/2)+
r_2^n,
\end{align}
where
$$r_2^n=\fe^{i\partial_x^2\tau/4} r_1^n+\int_{0}^\tau \frac{i}{4}\partial_x^2\fe^{i\partial_x^2s/4}\left(V_m+\alpha|v^n(\tau/2)|^2\right)^2v^n(\tau/2)
ds.$$
Under the assumption $s\geq5$, we have $\|r_2^n\|_{H^1}\leq C\tau^3$ with $C$ depending on the $H^3$-norm of $V_m$ and $\psi_m$.
On the other hand, by the Duhamel formula of (\ref{nls trun}), we have
\begin{align}\label{splitting duhamel}
  \psi_m(t_{n+1})=\fe^{\frac12 i\tau\partial_x^2}
  \psi_m(t_n)+i\int_0^\tau g_\tau^n(\rho)d\rho,\quad n\geq0,
\end{align}
with $g_\tau^n(\rho):=\fe^{\frac12 i(\tau-\rho)\partial_x^2}[V_m+
\alpha|\psi_m(t_n+\rho)|^2]\psi_m(t_n+\rho)$. By plugging
$ \psi_m(t_{n}+\rho)=v^n(\rho)+i\int_0^\rho g_\rho^n(s)ds$
into (\ref{splitting duhamel}) we find
\begin{align*}
\psi_m(t_{n+1})=&v^n(\tau)+i\int_0^\tau\fe^{\frac12 i(\tau-\rho)\partial_x^2}\left[V_m+\alpha
|v^n(\rho)|^2\right]v^n(\rho)d\rho\\
&+i\int_0^\tau\fe^{\frac12 i(\tau-\rho)\partial_x^2}
\int_0^\rho\left[i (V_m+2\alpha|v^n(\rho)|^2)g_\rho^n(s)-i\alpha v^n(\rho)^2\overline{g_\rho^n(s)}\right]ds d\rho+r_3^n,\end{align*}
where clearly $\|r_3^n\|_{H^1}\leq C\tau^3$ with $C$ depending on the $H^1$-norm of $V_m,\psi_m$.
Note $g_\rho^n(s)
=[V_m+\alpha|v^n(\rho)|^2]v^n(\rho)+\mathcal{O}(\tau\partial_x^2)$ for $0\leq s\leq \rho\leq\tau$,
we further get
\begin{align*}
 \psi_m(t_{n+1})=&v^n(\tau)+i\int_0^\tau\fe^{\frac12 i(\tau-\rho)\partial_x^2}\left[V_m+
\alpha|v^n(\rho)|^2\right]v^n(\rho)d\rho\\
&-\int_0^\tau
\rho (V_m+\alpha |v^n(\rho)|^2)^2v^n(\rho)d\rho+r_3^n+r_4^n,
\end{align*}
where $\|r_4^n\|_{H^1}\leq C\tau^3$ with $C$ depending on the $H^3$-norm of $V_m,\psi_m$.
Subtracting the above from (\ref{split err eq1}) and then using the quadrature error formula of the midpoint rule, we find
\begin{align*}
 &\|\psi_m(t_{n+1})-\Psi_\frac{\tau}{2}^\mathrm{k}\circ \Psi_\tau^{\mathrm{p}}\circ
\Psi_\frac{\tau}{2}^\mathrm{k}(\psi_m(t_n))\|_{H^{1}}\\
\leq& \|r_2^n\|_{H^1}+\|r_3^n\|_{H^1}+\|r_4^n\|_{H^1}+
\tau^3\int_0^1\kappa(\theta) \|f''(\theta\tau)\|_{H^1}d\theta ,
\end{align*}
where $\kappa(\theta)$ is the Peano kernel of the midpoint rule and
$$f(\rho):=i\fe^{\frac12 i(\tau-\rho)\partial_x^2}\left[V_m+
\alpha|v^n(\rho)|^2\right]v^n(\rho)-
\rho (V_m+\alpha |v^n(\rho)|^2)^2v^n(\rho).$$
Therefore, we find in total
\begin{align}\label{temporal eq1}
 \|\psi_m(t_{n+1})-\Psi_\frac{\tau}{2}^\mathrm{k}\circ \Psi_\tau^{\mathrm{p}}\circ
\Psi_\frac{\tau}{2}^\mathrm{k}(\psi_m(t_n))\|_{H^{1}}\leq C\tau^3,\quad 0\leq n< T/\tau,
\end{align}
with $C$ depending on the $H^5$-norm of $V_m,\psi_m$.
By the uniform bounds of $V_m$ and $\psi_m$ in $H^s$ with $s\geq5$ from the assumption and the result of  Lemma \ref{lm1} and Remark \ref{remark1}, we find the error constant $C>0$ in (\ref{temporal eq1}) is eventually a constant depending only on  $\|\psi\|_{L^\infty((0,T)\times U;H^s)}$ and
$\|V\|_{L^\infty(U;H^s)}$.

Next, by the triangle inequality and the projection error (denote by $P_M$ the projection operator to cutoff the Fourier modes of a function larger than $M/2$) \cite{Shen}, we find
\begin{align*}
\|I_M\psi_m^{n+1}-\psi_m(t_{n+1})\|_{H^{1}}\leq \|I_M\psi_m^{n+1}-P_M\psi_m(t_{n+1})\|_{H^{1}}+Ch^{s-1},
\end{align*}
with $C$ here depends only on $\|\psi_m\|_{L^\infty((0,T)\times U_m;H^s)}$ and so only on $\|\psi\|_{L^\infty((0,T)\times U;H^s)}$.
Noting that $I_M\psi_m^{n+1}=\Psi_\frac{\tau}{2}^\mathrm{k}\circ I_M[\Psi_\tau^{\mathrm{p}}\circ
\Psi_\frac{\tau}{2}^\mathrm{k}(I_M\psi_m^n)]$ and $\Psi_\frac{\tau}{2}^\mathrm{k}$ is isometric in $H^1$, we find by the triangle inequality and (\ref{temporal eq1}),
\begin{align}\|I_M\psi_m^{n+1}-P_M\psi_m(t_{n+1})\|_{H^1}\leq \|I_M\Psi_\tau^{\mathrm{p}}\circ
\Psi_\frac{\tau}{2}^\mathrm{k}(I_M\psi_m^n)- P_M\Psi_\tau^{\mathrm{p}}\circ
\Psi_\frac{\tau}{2}^\mathrm{k}(\psi_m(t_n))\|_{H^1}+C\tau^3. \label{temporal eq2}
\end{align}
Then by an induction on the boundedness of $I_M\psi_m^{n}$ in $H^1$-space and the difference between the flows, we find with some $C$ depending only on $\|\psi\|_{L^\infty((0,T)\times U;H^s)}$ and
$\|V\|_{L^\infty(U;H^s)}$,
\begin{align*}
 &\left\|I_M\Psi_\tau^{\mathrm{p}}\circ
\Psi_\frac{\tau}{2}^\mathrm{k}(I_M\psi_m^n)- P_M\Psi_\tau^{\mathrm{p}}\circ
\Psi_\frac{\tau}{2}^\mathrm{k}(\psi_m(t_n))\right\|_{H^1}\\
\leq &\|I_M\psi_m^{n}-P_M\psi_m(t_{n})\|_{H^1}+C\tau\|I_M\psi_m^n-P_M\psi_m(t_n)\|_{H^1}+C\tau h^{s-1}.
\end{align*}
Plugging it into (\ref{temporal eq2}), summing the inequalities up till $n=0$ and noting
$I_M\psi_m^{0}=I_M\psi_m(0)$, the Gronwall inequality gives
\begin{align*}
 \|I_M\psi_m^{n+1}-\psi_m(t_{n+1})\|_{H^1}\leq C(\tau^2+ h^{s-1}),
\end{align*}
for some $C$ depending only on $T,$ $\|\psi\|_{L^\infty((0,T)\times U;H^s)}$ and
$\|V\|_{L^\infty(U;H^s)}$.
Thus, there exist constants $\tau_0,h_0>0$ uniformly in $m$ such that for $\tau\leq\tau_0$ and $h\leq h_0$, $\|I_M\psi_m^{n+1}\|_{H^1}\leq \|\psi\|_{L^\infty((0,T)\times U;H^s)}+1$ for all $t_{n+1}\leq T$.
\end{proof}

With all the obtained results so far, we are now ready to get the full error bound stated in our main theorem in Section \ref{subsec:main} for the QMC-TSFP scheme (\ref{scheme}). The proof is given as follows.

\noindent\emph{Proof of the Theorem \ref{thm}:} By the triangle inequality, we have
\begin{align}
&\mathbb{E}^{\Delta}\left[\left|\mathbb{E}[\mathcal{G}](t_n)-Q_{m,N}[G(|I_M\psi_m^n|^2)]\right|^2\right]\nonumber\\
\leq& 2\mathbb{E}^{\Delta}\left[\left|\mathbb{E}[\mathcal{G}](t_n) - Q_{m,N}[\mathcal{G}^m](t_n) \right|^2\right]
+2\mathbb{E}^{\Delta}\left[\left|Q_{m,N}[G(|I_M\psi_m^n|^2)]- Q_{m,N}[\mathcal{G}^m](t_n) \right|^2\right]\nonumber\\
\leq &C(m^{-\chi}+ N^{\delta - 1})^2+2\mathbb{E}^{\Delta}\left[\left|Q_{m,N}[G(|I_M\psi_m^n(\xi,\cdot)|^2)-G(|\psi_m(t_n, \xi,\cdot)|^2)]\right|^2\right],\label{proofthm eq1}
\end{align}
where we used Proposition \ref{prop1}. By (\ref{QMC1}) and Assumption \ref{assump}, we find
\begin{align*}
  &\mathbb{E}^{\Delta}\left[\left|Q_{m,N}[G(|I_M\psi_m^n(\xi,\cdot)|^2)-G(|\psi_m(t_n, \xi,\cdot)|^2)]\right|\right] \\
  \leq &\sup_{\xi\in U_m}\left|G(|I_M\psi_m^n(\xi,\cdot)|^2)-G(|\psi_m(t_n, \xi,\cdot)|^2)\right|\\
  \leq &\|G\|_{H^1(D)'}\sup_{\xi\in U_m}
  \left\||I_M\psi_m^n(\xi,\cdot)|^2-|\psi_m(t_n, \xi,\cdot)|^2\right\|_{H^1}.
\end{align*}
Then, by the error estimate in Proposition \ref{prop2} which is uniform in $\xi$ and $m$, we  plug the above into (\ref{proofthm eq1}) to get
$$\sqrt{\mathbb{E}^{\Delta}\left[\Big|\mathbb{E}[\mathcal{G}](t_n)-Q_{m,N}[G(|I_M\psi_m^n|^2)]\Big|^2
\right]}\leq C
(\tau^2+h^{s-1}+m^{-\chi}+N^{\delta - 1}),\ t_n\in[0, T],
$$
for some constant $C>0$ independent of $m$. This completes the proof. \qed

\section{Numerical examples}\label{sec:result}
In this section, we carry out numerical experiments for the QMC-TSFP scheme. We shall test its accuracy in the random space, physical space and time to verify the theoretical result. The generating vector (\ref{lattice}) for QMC is constructed by the code in \cite{Kuoweb} (see also \cite{focm}). Some simulations on the wave propagation are provided in the end.

\subsection{Convergence test}

\begin{example}[Convergence in random space]
\label{example: convergence of QMC}
Consider \eqref{nls trun} with $\alpha = 1, \bT = [-\pi, \pi]$, the initial data
$
\psi_{in}(x)= \sqrt{8/\pi} \exp(-8x^2)
$
and the random potential
\begin{align} \label{eq: potential of convergence test}
V_m(\bxi, x) = 1 + \sum_{j = 1}^m \frac{1}{j^2} \xi_j \cos(j x),
\end{align}
where $\{\xi_j\}$ are i.i.d. uniformly distributed random variables on $[-1, 1]$.  Firstly, we test and compare the performance of the Monte Carlo (MC) method and the QMC method. One set of the numerical solutions are computed by TSFP combined with  MC  under $\tau = 10^{-4}, h = \frac{\pi}{64}$ and $N_{\text{MC}}$ samples $\{\xi_{\text{MC}}^{(p)}: p = 1, \ldots, N_{\text{MC}}\}$. The other set of solutions are computed by QMC-TSFP  under $\tau = 10^{-4}, h = \frac{\pi}{64}$ and $RN_{\text{QMC}}$ samples $\{\xi_{\text{QMC}}^{(r, p)}: r = 1, \ldots, R, p = 1, \ldots, N_{\text{QMC}}\}$, where $R$ is the number of shifts and $N_{\text{QMC}}$ is the number of samples for each  shift. To have a fair comparison between these two methods, we set $N_{\text{tot}} = N_{\text{MC}} = R N_{\text{QMC}}$. We choose $R = 10$ and $N_{\text{QMC}} = 2^{10}, 2^{11}, \ldots, 2^{16}$. The final time is fixed as $t_n=T = 1$.

\begin{figure}[t!]
  \centering
  \subcaptionbox{$L^2$ relative error \eqref{eq: L2 relative error of convergence test} for $m = 6$ \label{fig: QMC convergence for m = 6}}{\includegraphics[width=0.48\textwidth]{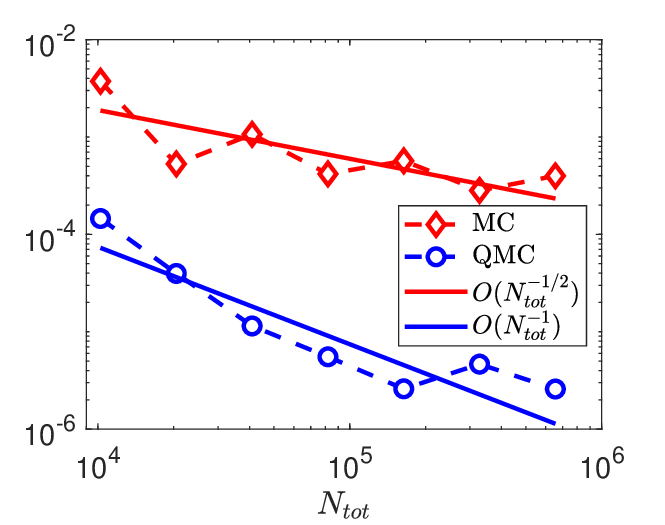}}
  \subcaptionbox{RMS \eqref{eq: RMS error for MC} and \eqref{eq: RMS error for QMC} for $m = 16$ \label{fig: QMC convergence for m = 16}}{\includegraphics[width=0.48\textwidth]{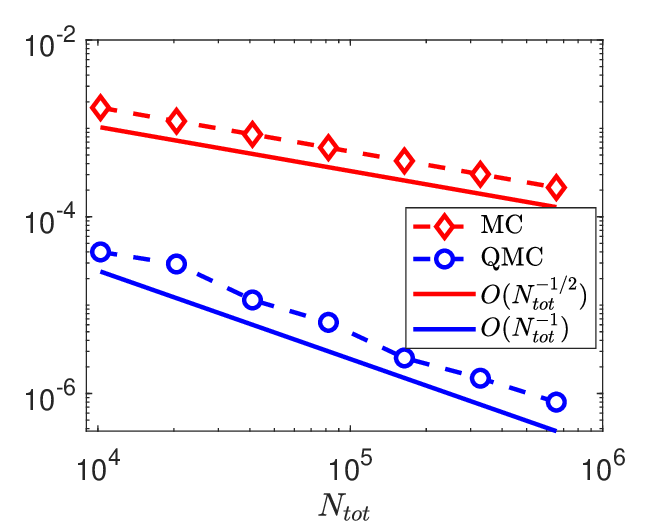}}
  \caption{Convergence of QMC-TSFP in random space.}
  \label{fig: QMC convergence}
\end{figure}

For $m$, we consider two choices: $m = 6$ and $m = 16$. For $m = 6$, we compute the reference solution using TSFP combined with the multidimensional stochastic collocation method \cite{babuvska2007stochastic} with $\tau = 10^{-4}, h = \frac{\pi}{64}$ and $15$ Gauss-Legendre points in each of the $m$ dimensions of $\xi$. We consider the $L^2$ relative error of the expectation of the density
\begin{align} \label{eq: L2 relative error of convergence test}
\text{$L^2$ relative error} = \frac{\Vert \mathbb{E}_{\text{num}}[\vert \psi_{\text{num}}(x) \vert^2] - \mathbb{E}_{\text{ref}}[\vert \psi_{\text{ref}}(x) \vert^2] \Vert_{L^2}}{\Vert \mathbb{E}_{\text{ref}}[\vert \psi_{\text{ref}}(x) \vert^2] \Vert_{L^2}},
\end{align}
where for the MC method
\begin{align}
\mathbb{E}_{\text{num}}[\vert \psi_{\text{num}}(x) \vert^2] = \frac{1}{N_{\text{MC}}} \sum_{p = 1}^{N_{\text{MC}}} \vert I_M \psi^n_m(\xi_{\text{MC}}^{(p)}, x) \vert^2,
\end{align}
and for the QMC method
\begin{align}
\mathbb{E}_{\text{num}}[\vert \psi_{\text{num}}(x) \vert^2] = \frac{1}{R N_{\text{QMC}}} \sum_{r = 1}^R \sum_{p = 1}^{N_{\text{QMC}}} \vert I_M \psi^n_m(\xi_{\text{QMC}}^{(r, p)}, x) \vert^2,
\end{align}
with $n = T / \tau$ and $M = 2\pi / h$. The result for $m = 6$ is shown in Figure \ref{fig: QMC convergence for m = 6}. For $m = 16$,  a reference solution is not easy to obtain. Here we consider a physical observable given by the following functional $G$:
\begin{align} \label{eq: linear functional of second spatial moment}
G(|\psi(t, \xi, x)|^2) = \int_{\bT} |x|^2 |\psi(t, \xi, x)|^2 dx.
\end{align}
Note that $G(|\psi|^2)$ is the second spatial moment of the density function $|\psi|^2$, which is important in checking {the localization phenomenon} (see also Example \ref{example: Anderson localization}). For both MC and QMC, we consider the root-mean-square (RMS) error $$RMS:=\sqrt{\mathbb{E}^{\Delta}\left[\Big|\mathbb{E}[\mathcal{G}](t_n)-Q_{m,N}[G(|I_M\psi_m^n|^2)]\Big|^2
\right]}.$$
As given in \cite{Sloan Act}, an unbiased estimator of the RMS for MC reads
\begin{align} \label{eq: RMS error for MC}
RMS\approx & \sqrt{ \frac{1}{N_{\text{MC}}(N_{\text{MC}} - 1)} \sum_{p = 1}^{N_{\text{MC}}} \left( G(|I_M \psi^n_m(\xi_{\text{MC}}^{(p)})|^2) - Q^{\text{MC}}_{m, N_{\text{MC}}}[G(|I_M \psi^n_m|^2)] \right)^2 },
\end{align}
where $Q^{\text{MC}}_{m, N_{\text{MC}}}[G(|I_M \psi^n_m|^2)] = \sum_{p = 1}^{N_{\text{MC}}} G(|I_M \psi^n_m(\xi_{\text{MC}}^{(p)})|^2)$, and the one for QMC reads
\begin{align} \label{eq: RMS error for QMC}
RMS \approx & \sqrt{ \frac{1}{R(R - 1)} \sum_{r = 1}^{R} \left( Q^{(r)}_{n, N_{\text{QMC}}}[G(|I_M \psi^n_m|^2)] - \overline{Q}_{m, N_{\text{QMC}, R}}[G(|I_M \psi^n_m|^2)] \right)^2 },
\end{align}
where $Q^{(r)}_{n, N_{\text{QMC}}}$ and $\overline{Q}_{m, N_{\text{QMC}, R}}$ are defined in \eqref{eq: QMC approximation using one single shift} and \eqref{eq: final QMC approximation} respectively. The result for $m = 16$ is shown in Figure \ref{fig: QMC convergence for m = 16}.

From the results in Figure \ref{fig: QMC convergence}, we observe approximately $\mathcal{O}(N_{\text{tot}}^{-1/2})$ convergence for MC and approximately $\mathcal{O}(N_{\text{tot}}^{-1})$ convergence for QMC. We remark that the decaying rate of the potential (\ref{eq: potential of convergence test}) in this example is in fact weaker than what is imposed in Assumption \ref{assump2}. The almost-linear  convergence rate of QMC is still observed, and this indicates that the condition on the potential for analysis might be relieved.


\end{example}

\begin{example}[Convergence in time \& physical space]
\label{example: convergence of TSFP}
Next, we test the accuracy of QMC-TSFP in time and physical space by considering the same setting  in Example \ref{example: convergence of QMC} with $m = 6$, where the reference solution is obtained as before. To investigate the convergence in time, we fix $R = 10, N_{\text{QMC}} = 2^{20}, h = \frac{\pi}{64}$ in QMC-TSFP and compute the numerical solution by $\tau = \frac{1}{40}, \frac{1}{80}, \frac{1}{160}, \frac{1}{320}$. For the convergence in space, we fix $R = 10, N_{\text{QMC}} = 2^{20}, \tau = 10^{-4}$ and take $h = \frac{\pi}{4}, \frac{\pi}{8}, \frac{\pi}{16}, \frac{\pi}{32}$. We still consider the $L^2$ relative error \eqref{eq: L2 relative error of convergence test} at $t_n=1$. Temporal convergence is shown in Figure \ref{fig: temporal convergence for m = 6}, and spatial convergence is shown in Figure \ref{fig: spatial convergence for m = 6}. Clearly from Figure \ref{fig: TSFP convergence}, we can observe the second-order convergence rate in time and the exponential convergence rate in space.

\begin{figure}[t!]
  \centering
  \subcaptionbox{$L^2$ relative error \eqref{eq: L2 relative error of convergence test} in time \label{fig: temporal convergence for m = 6}}{\includegraphics[width=0.48\textwidth]{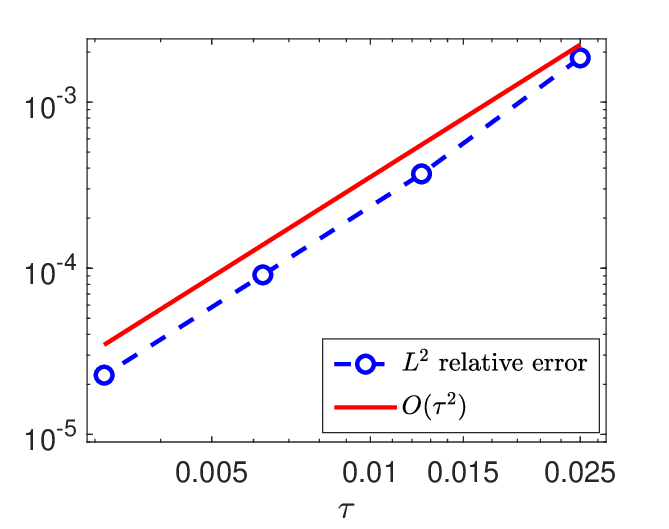}}
  \subcaptionbox{$L^2$ relative error \eqref{eq: L2 relative error of convergence test} in space \label{fig: spatial convergence for m = 6}}{\includegraphics[width=0.48\textwidth]{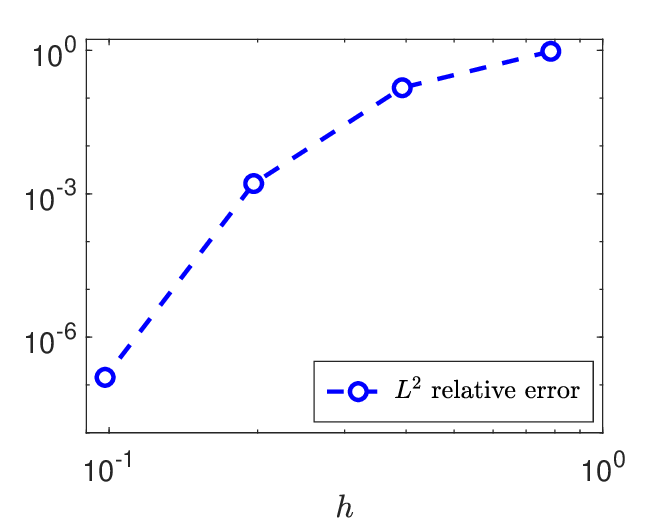}}
  \caption{Convergence of QMC-TSFP in time and physical space.}
  \label{fig: TSFP convergence}
\end{figure}

\end{example}

\subsection{{Simulation of wave propagation}}

\begin{example}
\label{example: Anderson localization}
{
Consider \eqref{nls trun} with $\alpha >0, \bT = [-6\pi, 6\pi]$,
$ \psi_{in}(x)= 1/\sqrt{\pi}\exp(-x^2)$
and
\begin{align} \label{eq: potential of Anderson localtization}
V_m(\bxi, x) = 1 + \sigma \sum_{j = 1}^{32} \frac{1}{j^{\frac{3}{2}}} \xi_j \cos(j x),
\end{align}
with $\{\xi_j\}$ the i.i.d. random variables uniformly distributed on $[-1, 1]$ and $\sigma \ge 0$ denoting the intensity of the randomness. Here the NLS (\ref{nls trun}) is equipped with a defocusing nonlinearity and the initial density function is exponentially localized in the physical domain. We solve (\ref{nls trun}) to simulate the wave propagations till the time $t_n=T = 10$ by using the QMC-TSFP method with $R = 10, N_{QMC} = 2^{16}, \tau = \frac{1}{400}$ and $h = \frac{3\pi}{256}$. The physical observable $G(|\psi(t, \xi, x)|^2)$ in \eqref{eq: linear functional of second spatial moment} is now considered to  measure the spreading of the density function in the physical space at time $t \in [0, T]$, and the expectation of $G(|\psi(t, \xi, x)|^2)$ is approximated by
\begin{align*}
\mathbb{E}[G](t) \approx A(t) := \overline{Q}_{m, N_{\text{QMC}, R}}[G(|I_M \psi^n_m|^2)](t)
\end{align*}
with $\overline{Q}_{m, N_{\text{QMC}, R}}$  in \eqref{eq: final QMC approximation}.
}

\begin{figure}[t!]
  \centering
  \subcaptionbox{Temporal behaviour of $A(t)$ \label{fig: temporal behaviour of second spatial moment}}{\includegraphics[width=0.48\textwidth]{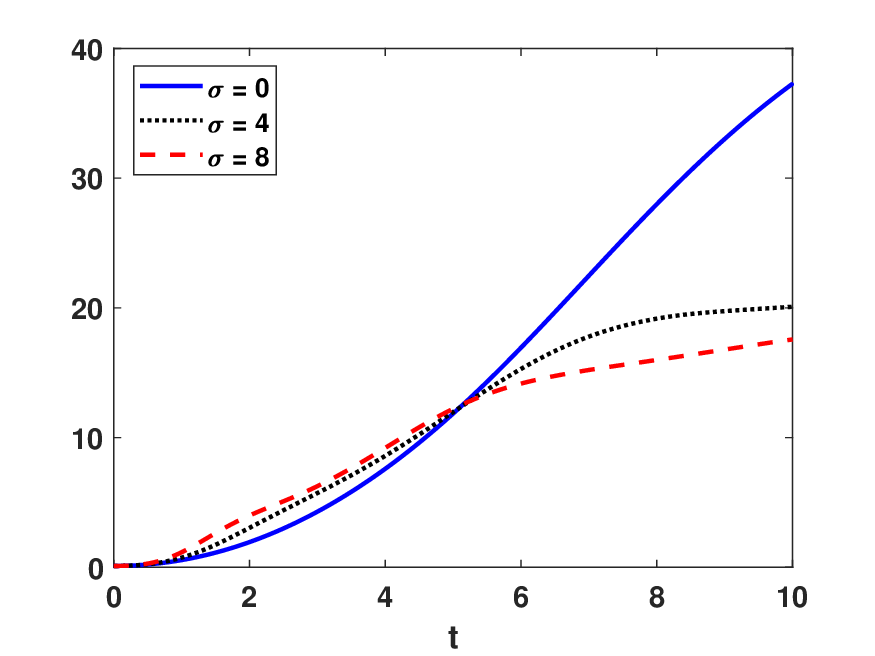}}
  \subcaptionbox{Distribution of $\Psi(x)$ at $t_n= 10$ \label{fig: expectation of density at final time}}{\includegraphics[width=0.48\textwidth]{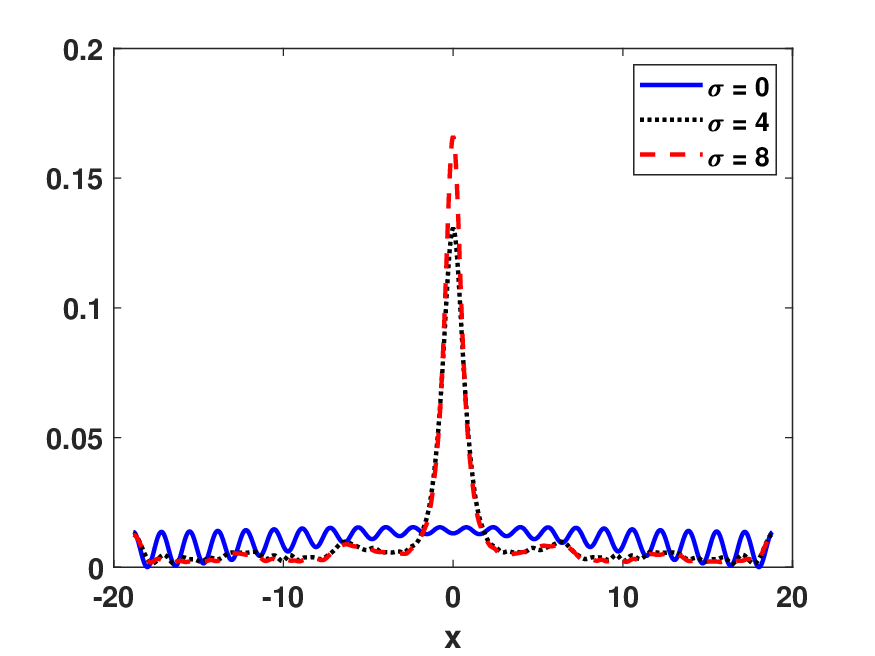}}
  \caption{Simulation of \textcolor{red}{wave propagation} in \textcolor{red}{random} NLS with $\alpha=1$.}
  \label{fig: Anderson localization}
\end{figure}

\begin{figure}[t!]
  \centering
  \subcaptionbox{$\sigma = 4$}{\includegraphics[width=0.48\textwidth]{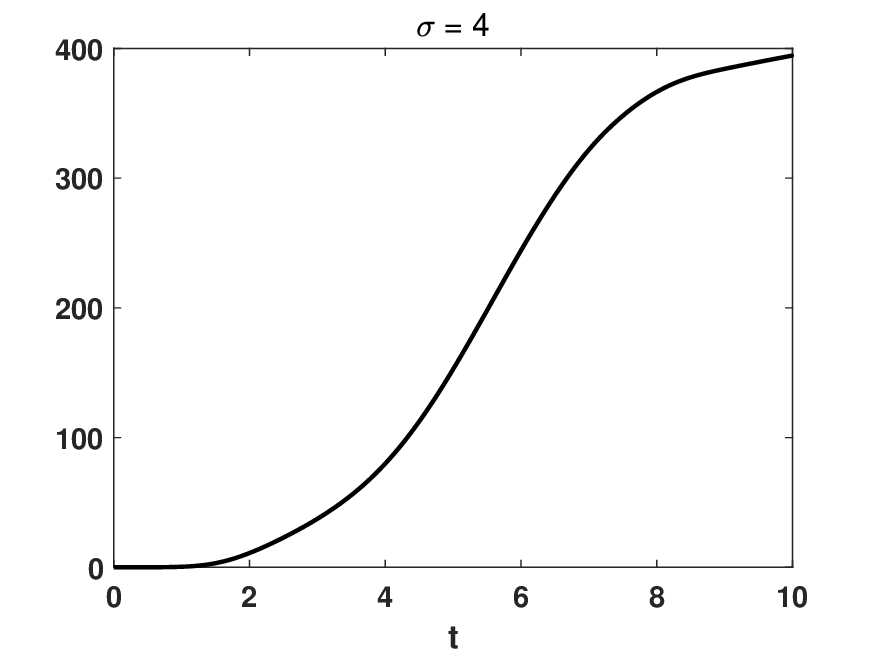}}
  \subcaptionbox{$\sigma = 8$}{\includegraphics[width=0.48\textwidth]{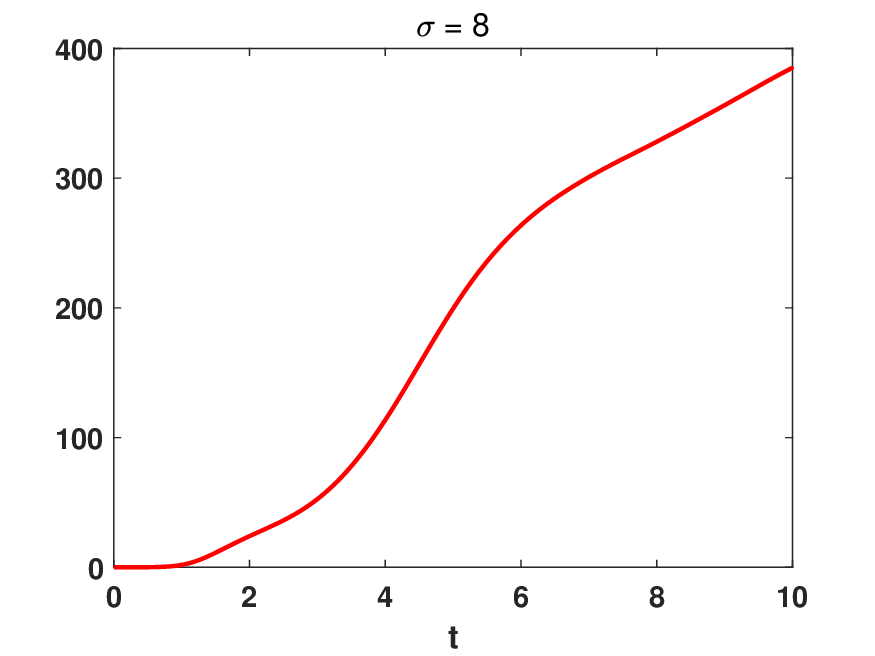}}
  \caption{Sample variances of $G(|\psi(t, \xi, x)|^2)$ with $\alpha = 1$.}
  \label{fig: Anderson localization variance of second spatial moment}
\end{figure}

\begin{figure}[t!]
  \centering
  \subcaptionbox{$\sigma = 4$}{\includegraphics[width=0.48\textwidth]{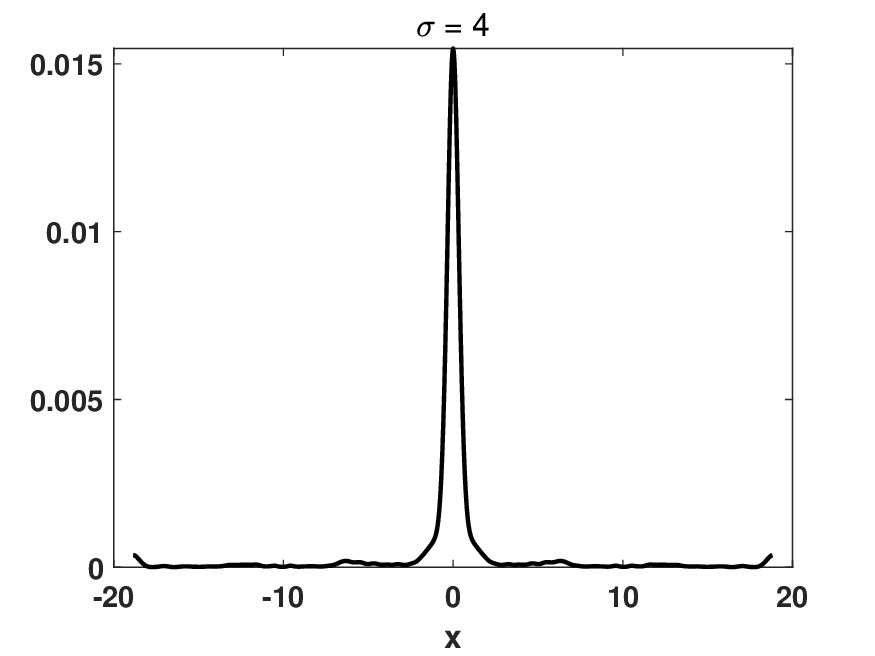}}
  \subcaptionbox{$\sigma = 8$}{\includegraphics[width=0.48\textwidth]{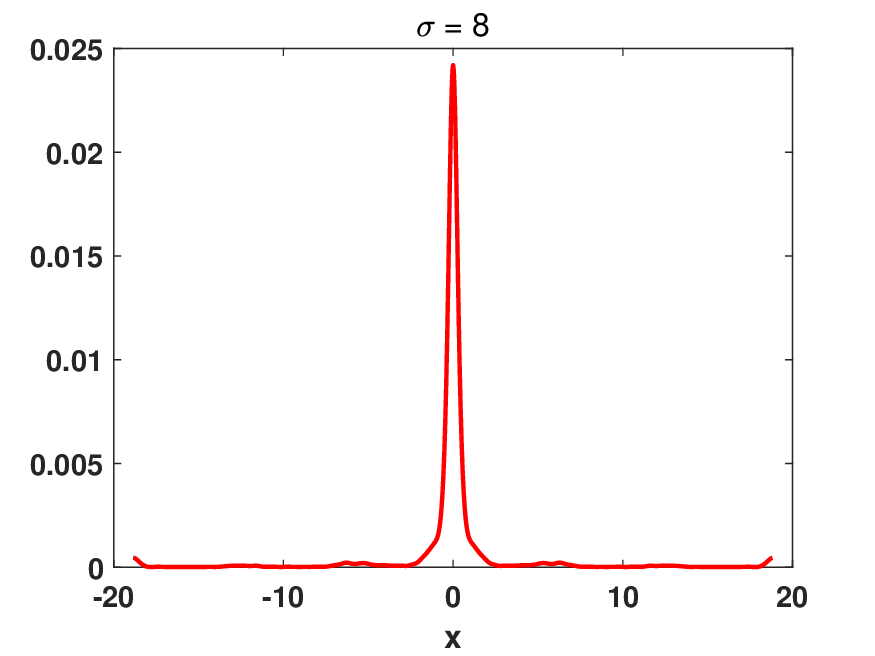}}
  \caption{Sample variances of $|\psi(T, \xi, x)|^2$ with $\alpha = 1$.}
  \label{fig: Anderson localization variance of density final time}
\end{figure}

\begin{figure}[t!]
  \centering
  \subcaptionbox{Temporal behaviour of $A(t)$}{\includegraphics[width=0.48\textwidth]
  {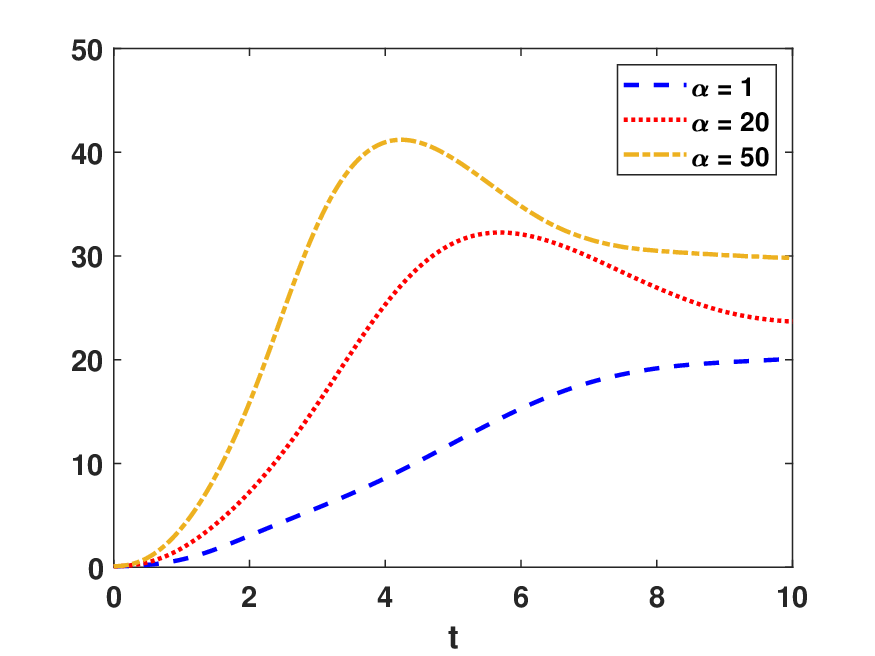}}
  \subcaptionbox{Distribution of $\Psi(x)$ at $t_n= 10$}{\includegraphics[width=0.48\textwidth]
  {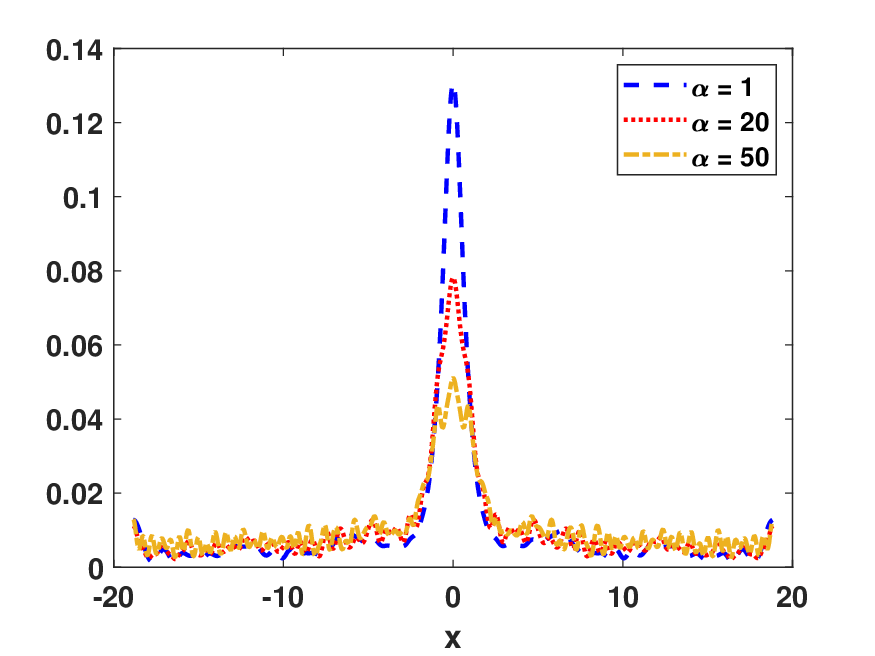}}
  \caption{Simulation of {wave propagation} in {random} NLS with $\sigma=4$.}
  \label{fig: Anderson localization nonlinearity}
\end{figure}

\begin{figure}[t!]
  \centering
  \subcaptionbox{$\alpha = 1$}{\includegraphics[width=0.32\textwidth]
  {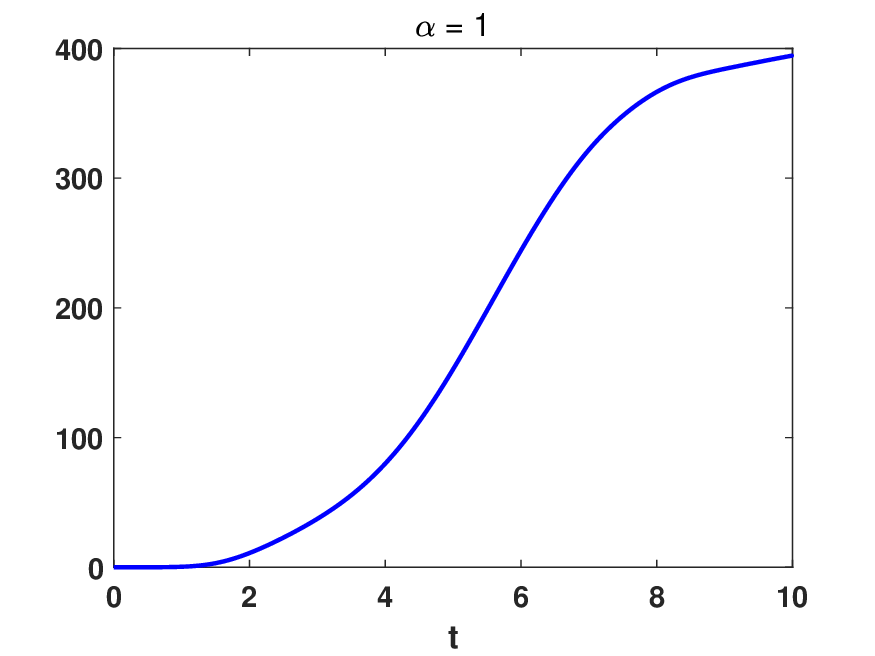}}
  \subcaptionbox{$\alpha = 20$}{\includegraphics[width=0.32\textwidth]
  {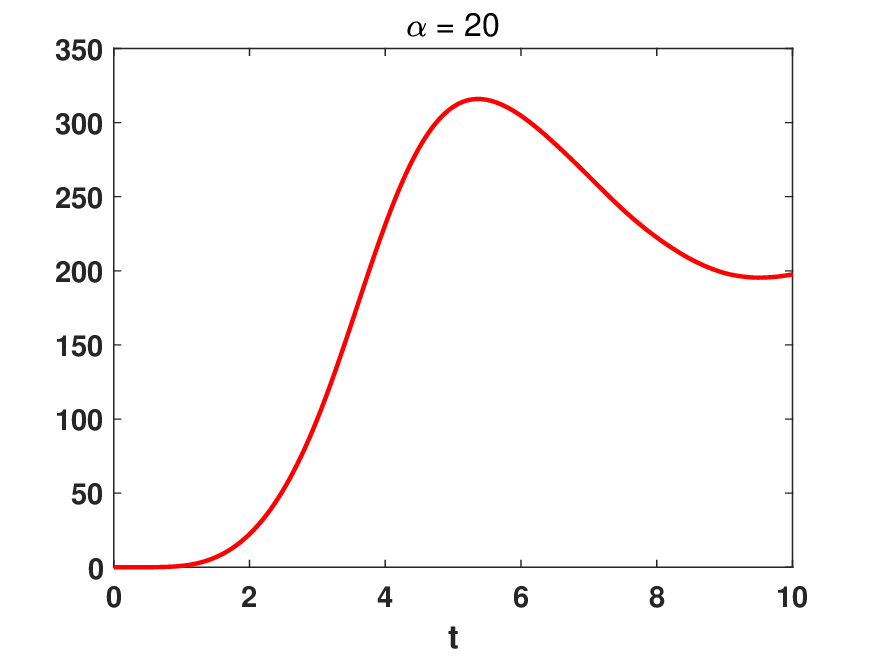}}
  \subcaptionbox{$\alpha = 50$}{\includegraphics[width=0.32\textwidth]
  {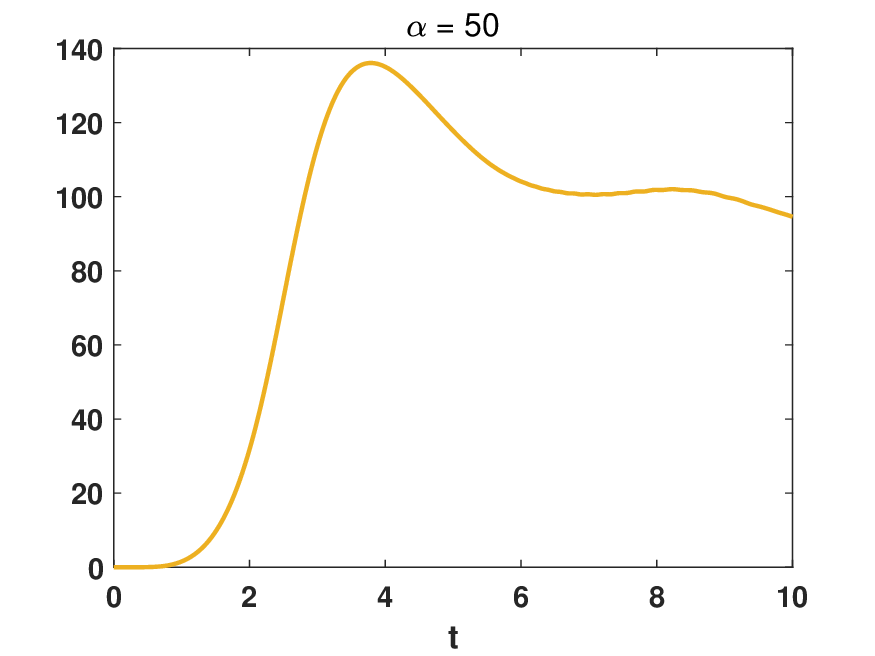}}
  \caption{Sample variances of $G(|\psi(t, \xi, x)|^2)$ with $\sigma = 4$.}
  \label{fig: Anderson localization nonlinearity variance of second spatial moment}
\end{figure}

\begin{figure}[t!]
  \centering
  \subcaptionbox{$\alpha = 1$}{\includegraphics[width=0.32\textwidth]
  {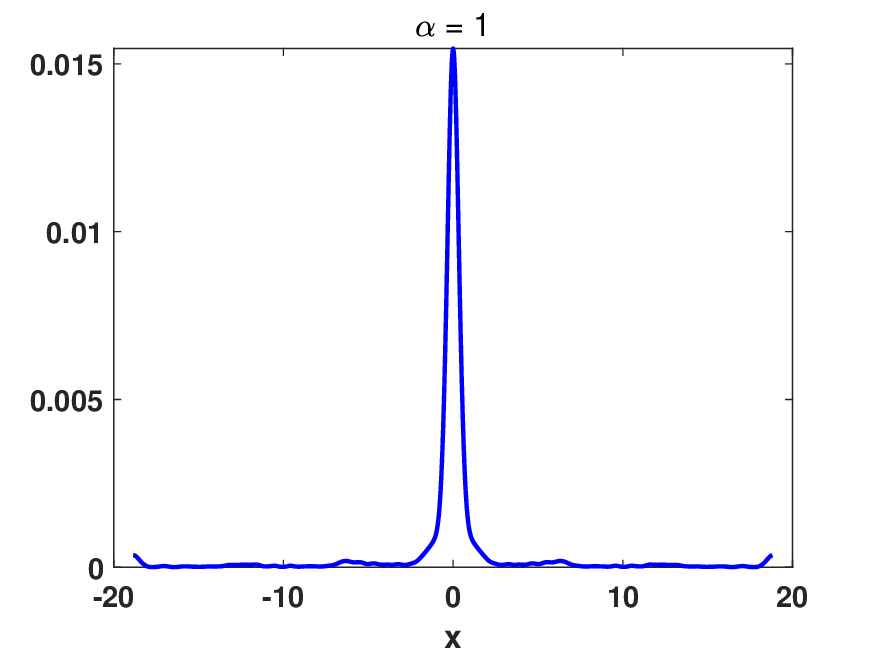}}
  \subcaptionbox{$\alpha = 20$}{\includegraphics[width=0.32\textwidth]
  {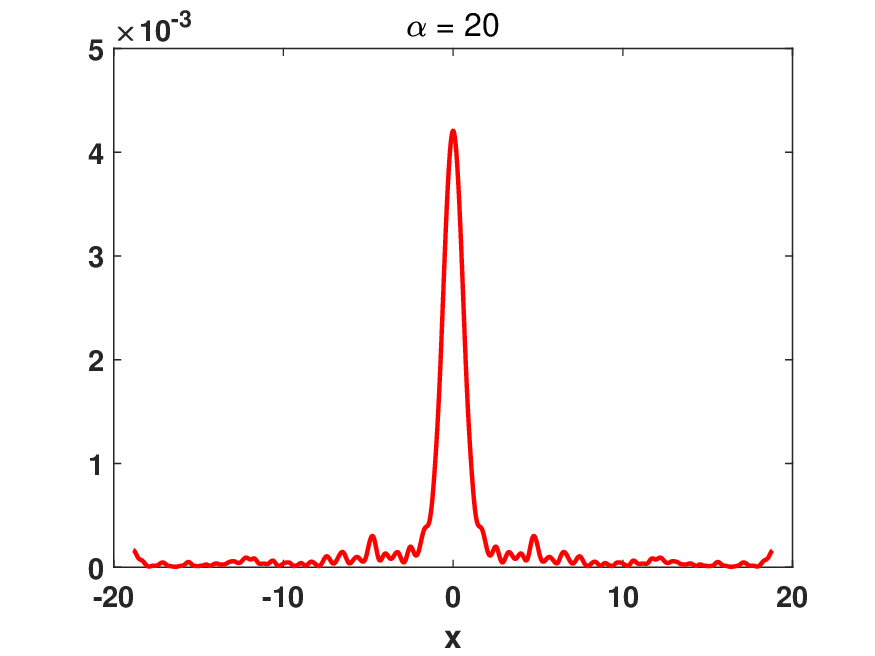}}
  \subcaptionbox{$\alpha = 50$}{\includegraphics[width=0.32\textwidth]
  {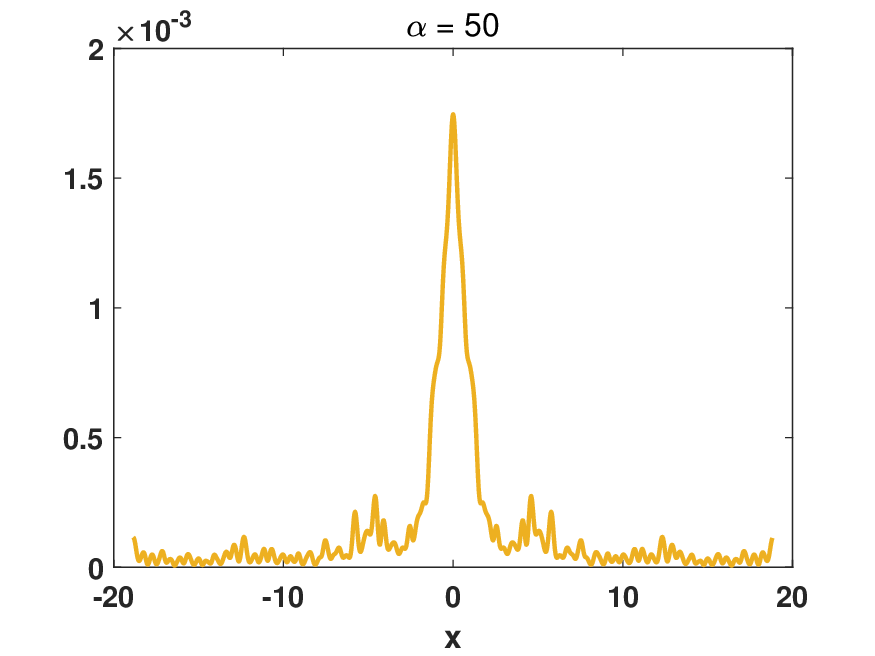}}
  \caption{Sample variances of $|\psi(T, \xi, x)|^2$ with $\sigma = 4$.}
  \label{fig: Anderson localization nonlinearity variance of density final time}
\end{figure}

{
Fixing $\alpha=1$, the temporal behaviour of $A(t)$ is shown in Figure \ref{fig: temporal behaviour of second spatial moment} for different $\sigma$, and the quantity
$$
\Psi(x):=\frac{1}{RN}\sum_{r=1}^{R}\sum_{p=1}^N |I_M \psi^n_m(\xi^{(r,p)},x)|^2,
$$
which approximates the expectation of the density $|\psi(T, \xi, x)|^2$ at the final time $T = 10$, is shown in Figure \ref{fig: expectation of density at final time}. In addition, the sample variances of the physical observable $G(|\psi(t, \xi, x)|^2)$ and the density $|\psi(T, \xi, x)|^2$ at the final time $T = 10$ are shown in Figures \ref{fig: Anderson localization variance of second spatial moment} and \ref{fig: Anderson localization variance of density final time}, respectively. In this case, we observe pure diffusion when randomness is absent in the potential \eqref{eq: potential of Anderson localtization}, and we observe localization of the expected density function within the computational time when randomness presents in the potential \eqref{eq: potential of Anderson localtization}.
When the randomness intensity $\sigma$ increases, the profile becomes more localized. By fixing $\sigma=4$ in (\ref{nls trun}), the corresponding results for different $\alpha$ are provided in Figures \ref{fig: Anderson localization nonlinearity}, \ref{fig: Anderson localization nonlinearity variance of second spatial moment} and \ref{fig: Anderson localization nonlinearity variance of density final time}. In this case, the profile of the expected density expands more as $\alpha$ becomes larger. This shows the defocusing nonlinear effect on the wave propagation.
}

\end{example}

\section{Conclusion}\label{sec:conclusion}

We proposed to combine the randomly shifted QMC lattice rule with the time-splitting Fourier pseudospectral method for solving the nonlinear Schr\"{o}dinger equation with {random potential}. We analysed the convergence in random space by using the technique of the weighted Sobolev space. The nonlinearity in the equation introduces difficulties in estimating the parametric regularity of the solution, and the physical observable considered here is typically a nonlinear functional of the solution. We propose sufficient conditions to show the existence of a QMC rule that can achieve the almost-linear and dimension-independent convergence rate for the expected value of the physical observable. The full error estimate of the scheme is established which also covers the convergence in time and  in the physical space. Numerical examples are presented to verify the theoretical result and simulations are done to investigate {the wave propagation} in the {random} NLS.


\section*{Acknowledgement}
\noindent
Z. Zhang is supported by Hong Kong RGC grant (projects 17300318 and 17307921), NSFC 12171406, a Seed Funding from the HKU-TCL Joint Research Centre for Artificial Intelligence, and the outstanding young researcher award of HKU (2020-21). X. Zhao is supported by NSFC 12271413, 11901440 and the Natural Science Foundation of Hubei Province 2019CFA007. The computations were performed using research computing facilities provided by Information Technology Services, the University of Hong Kong.

\appendix

\section{Some useful tools for studying NLS}
\label{appendix: useful tools}

{
To study the regularity of the solution of \eqref{nls} in the physical space and analyze the error resulting from dimension truncation, we need two useful tools introduced in \cite{tao2006nonlinear}. We briefly review these two tools in this section.}

{
\paragraph{\textbf{The bootstrap-type argument}} The bootstrap-type argument is an application of the bootstrap principle or the continuity method, which can be viewed as a continuous analogue of the principle of mathematical induction. The abstract bootstrap principle works in the following way.
\begin{proposition}
Let $I$ be a time interval, and for each $t \in I$ suppose we have two statements, a ``hypothesis'' $\mathbf{H}(t)$ and a ``conclusion'' $\mathbf{C}(t)$. Suppose we can verify the following four assertions:
\begin{enumerate}[(a)]
  \item (Hypothesis implies conclusion) If $\mathbf{H}(t)$ is true for some time $t \in I$, then $\mathbf{C}(t)$ is also true for for that time $t$.
  \item (Conclusion is stronger than hypothesis) If $\mathbf{C}(t)$ is true for some $t \in I$, then $\mathbf{H}(t^\prime)$ is true for all $t^\prime \in I$ in a neighbourhood of $t$.
  \item (Conclusion is closed) If $t_1, t_2, \ldots$ is a sequence of times in $I$ which converges to another time $t \in I$, and $\mathbf{C}(t_n)$ is true for all $t_n$, then $\mathbf{C}(t)$ is true.
  \item (Base case) $\mathbf{H}(t)$ is true for at least one time $t \in I$.
\end{enumerate}
Then $\mathbf{C}(t)$ is true for all $t \in I$.
\end{proposition}
The proof of this proposition can be found in \cite[Chapter 1.3]{tao2006nonlinear}. Moreover, for an illustration of how the bootstrap-type argument is applied to proving the local well-posedness of NLS, we refer readers to the proof of Proposition 3.8 in \cite{tao2006nonlinear}.}

{
\paragraph{\textbf{The product lemma}} As a useful tool from harmonic analysis, the product lemma reads as follows:
\begin{lemma}
For all $f, g \in H^s(\bT)$ with $s\geq1$, we have
\begin{align}
\| f g \|_{H^s(\bT)} \le C_{s} \left( \| f \|_{H^s(\bT)}  \| g \|_{L^{\infty}(\bT)} + \| f \|_{L^{\infty}(\bT)} \| g \|_{H^s(\bT)} \right),
\end{align}
where $C_{s}>0$ is a constant dependent on $s$.
\end{lemma}
\begin{proof}
Let $\bT=(0,2\pi)$ for simplicity, and the Fourier expansion reads $f(x)=\sum_{l\in\bZ}\hat{f}_l\fe^{ilx}$ where $\hat{f}_l$ denotes the Fourier coefficient.  For some $l\in\bZ$, we consider
\begin{align*}
  (1+l^2)^{s/2}|\widehat{(fg)}_l|=(1+l^2)^{s/2}\left|\sum_{l_1+l_2=l}\hat{f}_{l_1}\hat{g}_{l_2}\right|
  \leq\sum_{l_1+l_2=l}\left[1+(l_1+l_2)^2\right]^{s/2}|\hat{f}_{l_1}\hat{g}_{l_2}|.
\end{align*}
It is direct to check that
\begin{align*}
 \left[1+(l_1+l_2)^2\right]^{s/2}\leq 2^{s/2}\left[\sqrt{1+{l_1}^2}+\sqrt{1+{l_2}^2}\right]^s
 \leq C_s\left[(1+{l_1}^2)^{s/2}+(1+{l_2}^2)^{s/2}\right],
\end{align*}
where the last inequality used the convexity of the function $x^s$ for $x>0,\,s\geq1$. In the following, the constant $C_s>0$ may be different at each occurrence. Combining the above, we deduce
$
  (1+l^2)^{s/2}|\widehat{(fg)}_l|
  \leq C_s\sum_{l_1+l_2=l}(1+l_1^2)^{s/2}|\hat{f}_{l_1}\hat{g}_{l_2}|
  +C_s\sum_{l_1+l_2=l}(1+l_2^2)^{s/2}|\hat{f}_{l_1}\hat{g}_{l_2}|.
$
By taking the square, we find
 \begin{align*}
(1+l^2)^{s}|\widehat{(fg)}_l|^2
  \leq C_s\sum_{l_1+l_2=l}(1+l_1^2)^{s}|\hat{f}_{l_1}\hat{g}_{l_2}|^2
  +C_s\sum_{l_1+l_2=l}(1+l_2^2)^{s}|\hat{f}_{l_1}\hat{g}_{l_2}|^2.
  \end{align*}
  By Parseval's identity and note $\|fg\|_{H^s}^2=\sum_{l\in\bZ}(1+l^2)^s|\widehat{(fg)}_l|^2$, we get
  $$\|fg\|_{H^s}^2\leq C_s(\|f\|_{H^s}^2\|g\|_{L^2}^2+\|g\|_{H^s}^2\|f\|_{L^2}^2).$$
The assertion of the lemma is obtained by noting $\| f \|_{L^{2}}\leq C\| f \|_{L^{\infty}}$ and  Sobolev's embedding $\| f \|_{L^{\infty}} \le C_{s} \| f \|_{H^s}$.
For more general statement of  product lemma,  see e.g., \cite{tao2006nonlinear}.
\end{proof}
}
%
%
%


\end{document}